\documentclass[11pt,a4paper]{amsart}

\usepackage[english]{babel}
\usepackage{amssymb}
\usepackage{amsmath}
\usepackage{amsthm}
\usepackage{mathrsfs}
\usepackage{amsrefs}

\def\Xint#1{\mathchoice
{\XXint\displaystyle\textstyle{#1}}%
{\XXint\textstyle\scriptstyle{#1}}%
{\XXint\scriptstyle\scriptscriptstyle{#1}}%
{\XXint\scriptscriptstyle\scriptscriptstyle{#1}}%
\!\int}
\def\XXint#1#2#3{{\setbox0=\hbox{$#1{#2#3}{\int}$ }
\vcenter{\hbox{$#2#3$ }}\kern-.6\wd0}}

\def\dashint{\Xint-}

\renewcommand{\le}{\leqslant}
\renewcommand{\leq}{\leqslant}
\renewcommand{\ge}{\geqslant}
\renewcommand{\geq}{\geqslant}

\newcommand{\R}{\mathbb{R}}
\newcommand{\e}{\varepsilon}

\theoremstyle{plain}
\newtheorem{theorem}{Theorem}[section]
\theoremstyle{plain}
\newtheorem{proposition}[theorem]{Proposition}
\theoremstyle{plain}
\newtheorem{lemma}[theorem]{Lemma}
\theoremstyle{plain}
\newtheorem{corollary}[theorem]{Corollary}
\theoremstyle{plain}

\theoremstyle{plain}
\newtheorem{remark}[theorem]{Remark}
\theoremstyle{plain}
\newtheorem{definition}[theorem]{Definition}

\numberwithin{equation}{section}
\begin{document}
\title[A three-dimensional water wave problem]{One-dimensional symmetry for the solutions
	of a three-dimensional water wave problem}

\author[Eleonora Cinti]{Eleonora Cinti}
\address{E.C., 
Dipartimento di Matematica, Universit\`a di Bologna,
Piazza di Porta San Donato 5, 40126 Bologna, Italy}
\email{eleonora.cinti5@unibo.it}

\author[Pietro Miraglio]{Pietro Miraglio}

\address{P.M., Dipartimento di Matematica, Universit\`a di Milano, Via Cesare Saldini 50,
20133 Milan, Italy, Departament de Matem\`{a}tica Aplicada I,
Universitat Polit\`{e}cnica de Catalunya, Diagonal 647, 08028 Barcelona, Spain}
\email{pietro.miraglio@unimi.it}

\author[Enrico Valdinoci]{Enrico Valdinoci}
\address{E.V., 
Dipartimento di Matematica, Universit\`a di Milano, Via Cesare Saldini 50,
20133 Milan, Italy, and
School of Mathematics and Statistics,
University of Melbourne,
813 Swanston St, Parkville VIC 3010, Australia,
and
Istituto di Matematica Applicata e Tecnologie Informatiche,
Via Ferrata 1, 27100 Pavia, Italy}

\email{enrico@math.utexas.edu}
\thanks{It is a pleasure to thank Xavier Cabr\'{e}
for his very interesting comments and for very pleasant scientific
discussions.
The first author was supported by MINECO grant MTM2014-52402-C3-1-P, the
ERC Advanced Grant 339958 Complex Patterns for Strongly Interacting Dynamical Systems - COMPAT. B and is part of the Catalan research group 2014 SGR 1083.
The first and third authors were supported by the ERC Starting Grant 
``E.P.S.I.L.O.N.'' 
Elliptic Pde's and Symmetry of Interfaces and Layers for Odd Nonlinearities.
The third author was supported by the ARC Discovery Project grant ``N.E.W.''
Nonlocal Equations at Work.}

	\begin{abstract}
	We prove a one-dimensional symmetry result for a weighted Dirichlet-to-Neumann problem arising in a model for water waves in dimension $3$. More precisely we prove that minimizers and bounded monotone solutions depend on only one Euclidean variable.
	The analogue of this
	result for the 2-dimensional case (and without weights) was established in~\cite{DllV}. In this paper, a crucial ingredient in the proof is given by an energy estimate for minimizers obtained via a comparison argument.
	\end{abstract}

\maketitle

	\section{Introduction}
	
	\subsection{A water wave model}
	
	In this paper, we establish one-dimensional symmetry results
	for solutions of a Dirichlet to Neumann problem which arises in a model for water waves. 

	A classical water wave model is that of considering an ideal fluid with
	density~$\varrho$
	and velocity~$V$
	in the spatial region~$\R^2\times[0,H]$ (that is the ``sea'', which is
	assumed to be of depth~$H>0$). 
	For convenience, one can endow~$\R^2\times[0,H]$
	with coordinates~$x\in\R^2$ and~$y\in[0,H]$
	(we consider the level~$\{y=H\}$ as the ``bottom of the sea''
	and the level~$\{y=0\}$ as the ``surface of the sea''; in this notation,
	$y$ represents the ``depth of the sea'').
	The fact that the fluid
	is incompressible gives that~${\rm div }(\varrho V)=0$ and the irrotationality condition
	that~$V=\nabla v$ in~$\R^2\times(0,H)$. Assuming that the bottom of the sea
	is made of solid material, the impenetrability of the matter gives that the
	vertical velocity vanishes along~$\{y=H\}$.
	Then, given the values of~$v$ along the surface of the sea (and denoting
	such datum by~$u$), one is
	interesting in finding the vertical velocity on the surface, possibly weighted
	by the density of the fluid (this vertical velocity is, roughly
	speaking, responsible for the formation of a wave starting from the rest position
	of a ``flat sea''). The problem turns out to be linear with respect to the derivatives of the datum $u$ and semilinear in virtue of the nonlinearity $f(u)$, so
	it is convenient to denote the vertical velocity
	on the surface by~${\mathcal{L}}u$ (with a minus sign that we introduce for later convenience).
	In this setting, writing~$V:=(V_1,V_2,V_3)\in\R^3$,
	and denoting by~$v_y$ the derivative of~$v$ with respect to the vertical variable~$y$,
	the problem can be formulated as
	\begin{align}\label{1st}
	\begin{cases}
	0={\rm div }(\varrho V)={\rm div }(\varrho\nabla v ) 
	\quad&{\mbox{ in }}\R^2\times(0,H)\\
	0=V_3\big|_{y=H}= v_y\big|_{y=H}\quad&{\mbox{ on }}\R^2\times\{y=H\} \\
	u=v\big|_{y=0}\quad&{\mbox{ on }}\R^2\times\{y=0\} \\
	{\mathcal{L}}u=f(v)\quad&{\mbox{ on }}\R^2\times\{y=0\},
	\end{cases}
	\end{align}
{where ${\mathcal{L}} u=-\varrho v_y\big|_{y=0}$ and $f:\R\to\R$ is a given,
smooth
function.} When~$\varrho:=1$ and~$H\to+\infty$
(which is the case of a fluid with constant density
and an ``infinitely deep sea''), the problem in~\eqref{1st} is related to the square root of
the Laplacian, see e.g.~\cite{CS}. For finite values of~$H$
the operator described in~\eqref{1st} is nonlocal, but also not of purely fractional type.
Hence, in the sequel, we normalize the domain by setting~$H:=1$.

\subsection{Dirichlet to Neumann operators}

More specifically, in this paper, motivated by~\eqref{1st},
we consider the slab $\R^n\times[0,1]$ with coordinates $x\in\R^n, y\in[0,1]$
	and a function $u:\R^n\to\R$. We then consider $v=v(x,y)$ as the bounded
	extension of $u:\R^n\to\R$ in the slab~$\R^n\times (0,1)$ satisfying the following problem
	with density $\varrho(y)=y^a$, where $a\in (-1,1)$:
	\begin{align}
	\label{lvsis}
	\begin{cases}
	\mathrm{div}(y^a\nabla v)=0 \qquad &\text{in}\,\,\R^n\times (0,1)\\
	v(x,0)=u(x) \qquad &\text{on}\,\,\R^n\times \{y=0\}\\
	v_y(x,1)=0 \qquad &\text{on}\,\,\R^n\times \{y=1\}.
	\end{cases}
	\end{align}
Then, in view of the physical description in~\eqref{1st}, the problem in~\eqref{lvsis}
naturally leads to the study of the 
Dirichlet to Neumann operator $\mathcal{L}_a$ given by
	\[
	\mathcal{L}_a u(x)=-\displaystyle\lim_{y\rightarrow 0}y^av_y(x,y)
	\]
Notice that\footnote{In this paper we will always work with the extended problem satisfied by $v$ (see problem \eqref{mainsis} below), hence we do not actually need to define the operator $\mathcal L_a$ nor to discuss under which conditions we have uniqueness for solutions to \eqref{lvsis}.  We have chosen to introduce the Dirichlet-to-Neumann operator for the sake of completeness and to make a comparison with some well known related results for nonlocal equations involving the fractional Laplacian.}
the operator $\mathcal{L}_a$ is given by the operator $\mathcal{L}$ appearing 
in (\ref{1st}) for the choice $\varrho(y)=y^a$, which is the weight that
we consider throughout the paper.
We also observe that the operator $\mathcal L_a$ is closely
	related to the fractional Laplacian $(-\Delta)^s$ with $s=(1-a)/2$, see e.g.~\cite{CS},
	though it is not equal to any purely
	fractional operator. The case $a=0$, which corresponds to
	$v$ being the harmonic extension of $u$ in $\R^n \times (0,1)$, was considered in~\cite{DllV}, where the authors write explicitly the Fourier symbol of the operator $\mathcal L_0$ in this specific case and show that, for large frequencies, the Fourier symbol of $\mathcal L_0$ is asymptotic to the Fourier symbol of the half-Laplacian (observe that for $a=0$ we have $s=1/2$). 
	
	We study here the one-dimensional symmetry of certain bounded solutions
	to the problem
	\begin{equation}
	\label{eq12}
	\mathcal{L}_au=f(u) \qquad \text{in}\,\, \R^n,
	\end{equation}
	where $f\in C^{1,\gamma}(\R)$, with $\gamma>\max \{0,-a\}$, $n=3$ and $a\in(-1,1)$.
	
	{The results obtained in this paper
	extend a known result by
	de la Llave and the third author in~\cite{DllV}. In that paper,  one-dimensional symmetry of monotone solutions of (\ref{eq12}) is established when $n=2$ and $a=0$.} With the extension~\eqref{lvsis}, our problem can be formulated in the following local way, 
	{that we are going to consider throughout the paper}:
	\begin{align}
	\label{mainsis}
	\begin{cases}
	\mathrm{div}(y^a\nabla v)=0 \qquad &\text{in}\,\,\R^n\times (0,1)\\
	-\displaystyle\lim_{y\rightarrow 0}y^av_y=f(v)\qquad &\text{on}\,\,\R^n\times\{y=0\}\\
	v_y(x,1)=0 \qquad &\text{on}\,\,\R^n\times \{y=1\}.
	\end{cases}
	\end{align}
	{We observe that this formulation of our problem is exactly (\ref{1st}) for the general $n$-dimensional case, with $H=1$ and the weight $\varrho(y)=y^a$}. We will show that, if $v$ is a minimizer (in the sense of Definition \ref{minimizer} below) or a bounded monotone solution for problem~\eqref{mainsis} with $n=3$, then there exist a function $v_0:\R\times(0,1)\to\R$ and a
	vector~$\omega\in S^2$ such that
	\[
	v(x,y)=v_0(\omega\cdot x, y) \qquad \text{for every}\,\,x\in\R^3,\,y\in(0,1).
	\]
	In particular, the trace $u$ of $v$ on $\{y=0\}$ exhibits one-dimensional symmetry, i.e. it is a function of only one Euclidean variable.
	See also Chapter 3 of~\cite{BV} for additional discussions.
	
	It is interesting to point out that the results that we give here are new even
	in the case~$a:=0$, corresponding to uniform density of the fluid.
	Nevertheless, we provided a general setting for the problem
	in~\eqref{mainsis} and we believe that such generality is worthwhile for a series
	of reasons:
	\begin{itemize}
	 \item {F}rom a pure mathematical perspective,
	 weights of the type~$y^a$ belong to the Mouckenhoupt\footnote{As customary, one says
	 that a weight~$w$ belongs
	 to the Mouckenhoupt class~$A_2$ if
	 there exists~$C>0$ such that, for all balls~$ B$, 
it holds that 
$$ \dashint_B w ( x ) \,d x \;
\dashint_B\frac1{w ( x )} \,d x 
\le C,$$
with~$\dashint$ denoting average.
Roughly speaking, Mouckenhoupt weights may be singular or degenerate,
but they cannot be ``too singular or too degenerate'', in an integral sense.
Also, $w$ belongs to~$A_2$ if and only if so does~$1/w$.
}
	 class~$A_2$,
	 see~\cite{FKS,FJK}, which plays a special interest in the analysis of partial differential
	 equations with weights, since, in a sense, these weights constitute the fundamental
	 example of nontrivial, possibly singular or degenerate, weights, for which
	 a ``good elliptic theory'' is still possible;
	 \item With respect to fractional operators, it is important
	 to study different values of~$a$, corresponding to different values
	 of the fractional parameter (and, in this case, the value~$a:=0$ is
	 often a fundamental threshold dividing ``local'' and ``nonlocal'' behaviors
	 at large scales, see e.g. Theorem~1.5 in~\cite{MR2948285});
	 \item In applications, weights of the type~$y^a$ can model laminated
	 materials (see e.g.~\cite{MR2653742}) and, in the context of fluids, describe situations
	 in which the density of the fluid only depends on the depth;
	 \item In other situations, equations as in~\eqref{mainsis} can be related
	 to models in biological mathematics, in which~$v$ represents for instance
	 the density of a given population: in this setting, many real-world experiments
	 have confirmed that different populations exhibit anomalous diffusion,
	 and that the diffusion parameters vary from one species to another (see e.g.~\cite{Nature07}
	 and the references therein),
	 therefore it is relevant for concrete models to study nonlocal equations for all the parameter
	 values;
	 \item Most importantly, in our perspective, different values of~$a$
	 allow us more easily to (at least formally) interpolate between classical
	 partial differential equations (in a sense, corresponding to the case~$a\to1$)
	 and strongly nonlocal equations (corresponding to the case~$a\to-1$). Hence,
	 since the case~$a:=0$ is of course extremely important to address,
	 it is also crucial to comprise in the analysis all the values~$a\in(-1,1)$, so to develop a
	 much better intuition of the problem and to
	 permit the use of continuity
	 and bifurcation methods.
	 In this way, the investigation of nonlocal
	 problems produces results for classical questions which would have
	 not been available with other techniques. For instance,
	a very neat example in which fractional methods lead to new
and important results in classical cases is embodied by the recent
work~\cite{FS},
in which the Authors brilliantly exploit nonlocal tools developed e.g. in~\cite{SV}
and~\cite{023495}
to obtain symmetry result in a Peierls-Nabarro 
model;
\item The investigation of fractional problems
in the full range of the fractional exponent cases~$s\in\left(\frac12,1\right)$,
$s=\frac12$ and~$s\in\left(0,\frac12\right)$ is also important to understand
the different behaviors of the energy contributions (see e.g. \cite{MR2948285}).
As a matter of fact, typically, when~$s=\in\left(\frac12,1\right)$,
in spite of the nonlocal character of the problem,
the major contribution is ``of local type'', in the sense that it comes
from a very well delimited region of the space
in which ``all the action takes place''.
Conversely, when~$s\in\left(0,\frac12\right)$ the major contribution
comes ``from infinity'' and long-range interactions become predominant.
In this spirit, the case~$s=\frac12$ keeps a balance between these two
energy tendencies and, in fact, when~$s=\frac12$
the energy contributions typically ``repeat themselves at each dyadic scale'',
and, in practice,
this special additional invariance
often produces logarithmic energy terms that are characteristic
for the case~$s=\frac12$.
	\end{itemize}

	\subsection{Connection with the Allen-Cahn equation and a conjecture by De Giorgi}\label{DG}
	As we mentioned above, our operator $\mathcal L_a$ 
	is related to the $s$-Laplacian~$(-\Delta)^s$, for $s:=(1-a)/2$,
and the fractional
	Laplacian can be seen as a Dirichlet-to-Neumann operator
	for a local problem in the halfspace $\R^{n+1}_+$. 
	More precisely, Caffarelli and Silvestre in~\cite{CS} 
	proved that one can study a semilinear nonlocal problem of the form
	\begin{equation}
	\label{sL}
	(-\Delta)^su=f(u) \qquad \text{in}\,\,\R^n,
	\end{equation}	
	by studying the associated local problem
	\begin{equation}\label{CS}\begin{cases}
	\mathrm{div}(y^{1-2s}\nabla v)=0 \qquad &\text{in}\,\,\R^{n+1}_+\\
	-\displaystyle\lim_{y\rightarrow 0}y^av_y=f(u)\qquad &\text{on}\,\,\R^n\times\{y=0\}.
	\end{cases}
	\end{equation}
	The problem of finding one-dimensional symmetry results for monotone solutions to~(\ref{sL}) is the counterpart, in the fractional setting,
	of a celebrated conjecture stated in 1978 by E. De Giorgi about bounded
	and monotone solutions of the classical Allen-Cahn equation~$-\Delta u=u-u^3$, see~\cite{MR533166}.
	The nonlinearity $u-u^3$, whose primitive (up to a sign) has a double well potential structure, arises in the study of phase transitions problem. 
	
	In dimension $n=2$ the fractional De Giorgi conjecture has been
	proved in~\cite{CSM} for $s=\frac{1}{2}$ and in~\cite{YV, CY2, SV}
	for every~$s\in(0,1)$. The same result in dimension $n=3$ has been established
	by Cabr\'{e} and the first author in~\cite{CC1} and~\cite{CC2} with respectively $s=\frac{1}{2}$ and $s\in(\frac{1}{2},1)$. Then, Savin proved in~\cite{S} the conjecture in dimensions $4 \leq n\leq 8$ for $s\in(\frac{1}{2},1)$ and with the additional assumption
	\begin{equation}
	\label{Savin}
	\lim_{x_n\to\pm\infty}u(x',x_n)=\pm1.
	\end{equation}
	Very recently, the conjecture has been proved in dimension $n=3$ and with $s\in(0,\frac{1}{2})$
	independently and with different methods
	by Dipierro, Farina and the third author in~\cite{DFV} (using an improvement
	of flatness result by~\cite{XFAH}) and by Cabr\'{e}, Serra and the first author
	in~\cite{CCS} (by a different approach which relies on some sharp energy
	estimates and a blow-down convergence result for stable solutions).
	
	In another very recent result, Figalli and Serra proved in~\cite{FS} the conjecture for monotone solutions of the half-Laplacian in dimension four without assumption~(\ref{Savin}) and we plan to further investigate this new method in the setting of water waves.
	
	Focusing on the dimension that
	we take into account in this paper, i.e. $n=3$,
	we want to stress an important difference between the 
	water wave problem and the fractional De Giorgi conjecture. As mentioned above, a different approach is needed to prove the one-dimensional symmetry of solutions to~(\ref{sL}) when the parameter $s$ crosses the value $\frac{1}{2}$. This is due to fact that the optimal energy estimates for solutions of~(\ref{sL}) change depending whether $s$ is above or below $1/2$, as shown in ~\cite{CC2}. In particular, only when $s\in [1/2,1)$ these energy estimates are enough to apply a Liouville type result and hence to obtain one-dimensional symmetry.  As we are going to see, this does not happen in our case, since the framework is $\R^n\times(0,1)$ and the weight $y^a$
	is integrable between~$ 0$ and~$ 1$. This fact gives us some energy estimates that do not depend on $a$ and allows us to prove one-dimensional symmetry of certain solutions to~(\ref{eq12}) with the same method for all the powers $a\in(-1,1)$, namely $s\in(0,1)$.
	
	For similar results in further dimensions, both in the classical and in the nonlocal
	case, see also~\cite{AAC, AC, YV, CC2, CY2, CSM, SV, S, DSV}.
	
	We stress that these types of nonlocal or fractional problems
	usually present several sources of additional difficulties with respect
	to the classical cases, such as:
	\begin{itemize}
	 \item Lack of explicit barriers
	 and impossibility of performing straightforward calculations;
	 \item Slow decay of the solutions at infinity;
	 \item Long range interactions and contributions coming from infinity;
	 \item Infinite energy amounts;
	 \item Formation of new types of interfaces (such as ``nonlocal minimal surfaces'').
	\end{itemize}
	In general, we also stress that nonlocal operators may present
	important differences with respect to the classical ones,
	also at a very basic level (see e.g. the introductory discussion
	in Section~2.1 of~\cite{ABA}).
	
\subsection{Variational formulation}
	As one can easily observe, problem~\eqref{mainsis} has a variational structure. Let $B_R\subset \R^n$ denote the ball of radius $R$ centered at $0$, 
	and $C_R$ the cylinder \begin{equation}\label{CRCR}C_R:=B_R\times(0,1) .\end{equation}
	The (localized) energy functional associated to problem~\eqref{mainsis} is given by
	\[
	\mathcal{E}_R(v)=\int_{C_R}y^a\lvert \nabla v \rvert^2 \,dx\,dy + \int_{ B_R\times \{y=0\}} G(v)\,dx,
	\]
	where the associated potential $G$ is such that $G'=-f$.
	
	We can now give the definitions of \textit{minimizer} and of \textit{stable solution}  for problem~\eqref{mainsis} (problem~\eqref{eq12} respectively) in a standard way.
	
	\begin{definition}\label{minimizer}
		We say that a bounded $C^1( \R^n\times (0,1))$ function $v$  is a  minimizer for~\eqref{mainsis} if
		\[
		\mathcal{E}_R(v)\leq \mathcal{E}_R(w)
		\]
		for every $R>0$ and for every bounded competitor $w$ such that $v\equiv w$ on $\partial B_R\times (0,1)$.\\
		We say that a bounded $C^1(\R^n)$ function $u$ is a minimizer for~\eqref{eq12} if its extension $v$ satisfying~\eqref{lvsis} is a minimizer for~\eqref{mainsis}.
	\end{definition}

	\begin{definition}\label{STAB}
		We say that a bounded solution $v$ of~\eqref{mainsis} is \textit{stable} if 
		\[
		\int_{\R^n\times[0,1]}y^a\lvert \nabla \xi \rvert^2\,dx\,dy-\int_{\R^n\times \{y=0\}}f'(u)\xi^2\,dx \geq0
		\]
		for every function $\xi\in C_0^1(\R^n\times [0,1])$. \\
		
		We say that a bounded function $u$ is a \textit{stable} solution for~\eqref{eq12} if its extension $v$ satisfying~\eqref{lvsis} is a stable solution for~\eqref{mainsis}.
	\end{definition}
	
	Clearly, if $v$ is a minimizer for~\eqref{mainsis} then, in particular, it is a stable solution. As we will observe
	later on in Section~\ref{stable.solutions} (see Remark~\ref{remark.mono}), also a monotone solution is stable, hence stability is a weaker notion of both minimality and monotonicity.
	
	The one-dimensional symmetry result in two dimensions for the particular case $a=0$
	obtained in~\cite{DllV} follows as a corollary of a more general result (see Theorem 1 in~\cite{DllV}), which states that a bounded monotone solution satisfying a certain energy estimate is necessarily one-dimensional.
	
	More precisely, Theorem 1 in~\cite{DllV} requires the existence of a positive constant $C$ such that:
	\begin{equation}
	\label{890}
	\int_{C_R}\lvert\nabla_x v(x,y)\rvert^2\,dx\,dy\leq CR^2,
	\end{equation}
	where $\nabla_x$ denotes the gradient in the $x$-variables.
	
	This is trivially true in the case $n=2$, thanks to the fact that the gradient of $v$ is bounded, by standard elliptic estimates (see~\cite{GT}).
	
	\subsection{Main results}
	
	The first result of this paper generalizes Theorem 1 of~\cite{DllV}
	to the class of stable solutions.
	As it will become clear from the proof, this generalization in itself
	is not too difficult but it will be technically
	crucial for the purpose of this paper, and in particular
	to prove the one-dimensional symmetry of monotone solutions in $\R^3$.	
	Therefore, we state explicitly this result as follows:
	
	\begin{theorem}\label{DllV}
		Let $f\in C^{1,\gamma}(\R)$, with $\gamma>\max \{0,-a\}$, and let $v$ be a bounded and stable solution of~(\ref{mainsis}).
		
		Suppose that there exists $C>0$ such that
		\begin{equation}\label{EB}
		\int_{C_R}y^a |\nabla_x v(x,y)|^2\,dx\,dy\le
		CR^2
		\end{equation}
		for any $R\ge 2$.
		
		Then, there exist $v_0:\R\times(0,1)\rightarrow\R$ and $\omega
		\in {\rm S}^{n-1}$ such that
		\begin{equation}
		\label{SYY}
		v(x,y)=v_0 (\omega\cdot x,y)\qquad{\mbox{
				for any $(x,y)\in\R^{n+1}_+$.}}
		\end{equation}
In particular, the trace $u$ of $v$ on $\{y=0\}$ can be written as $u(x)=u_0(\omega\cdot x)$.
		
		Moreover, $u_0'>0$ or $u_0'\equiv0$.
	\end{theorem}

From this Theorem, we can directly obtain as a Corollary the one-dimensional symmetry of stable solutions of (\ref{mainsis}), when $n=2$ and for every $a\in(-1,1)$. This extends the result of de la Llave and the third author in \cite{DllV}, in which they consider $n=2$, $a=0$ and $v$ as a monotone solution of (\ref{mainsis}).

	\begin{corollary}
		\label{2D}
		Let $f\in C^{1,\gamma}(\R)$, with $\gamma>\max \{0,-a\}$ and let $n=2$. Assume that $v$ is a bounded stable solution for problem~(\ref{mainsis}).
		Then, there exist $v_0:\R\times(0,1)\to\R$ and $\omega\in S^2$ such that:
		\[
		v(x,y)=v_0(\omega\cdot x,y) \qquad {\mbox{for all }} (x,y)\in\R^3\times(0,1).
		\]
		
		In particular, the trace $u$ of $v$ on $\{y=0\}$ can be written as $u(x)=u_0(\omega\cdot x)$.
	\end{corollary}
	
	It is an open problem whether the energy estimate~\eqref{EB} holds for stable solutions when $n=3$. In the following two results, we establish it for minimizers and for monotone solutions that, as observed before, are in particular stable solutions.
	
	Next result is an energy estimate for minimizers in any dimension $n$.
	This type of results are essential in order to check energy
	conditions as in~\eqref{EB} and so apply
	Theorem~\ref{DllV}.
	
	\begin{theorem}[\textbf{Energy estimate for minimizers}]
		\label{globmin}
		Let $f\in C^{1,\gamma}(\R)$, with $\gamma>\max \{0,-a\}$, and let $v$ be a bounded minimizer for problem~\eqref{mainsis}. 
		
		Then, we have
		\begin{equation}
		\label{estimate.min}
		\mathcal{E}_R(v) \leq C R^{n-1},
		\end{equation}
		for any $R\ge 2$.
		
	\end{theorem}	
	
	When $n=3$ we can prove the same estimate for bounded solutions whose traces on $\{y=0\}$ are monotone in some direction.
	
	\begin{theorem}[\textbf{Energy estimate for monotone solutions for $n=3$}]
		\label{thmono}
		Let $f\in C^{1,\gamma}(\R)$, with $\gamma>\max \{0,-a\}$, and let $v$ be a bounded solution of~(\ref{mainsis}) with $n=3$ such that its trace $u(x)=v(x, 0)$ is monotone in some direction. 
		
		Then, we have
		\begin{equation}
		\label{estimate.mono}
		\mathcal{E}_R(v) \leq C R^2,
		\end{equation}
		for any $R\ge 2$.
	\end{theorem}	
	
	{As mentioned in Subsection \ref{DG}, it is worth to stress
	a crucial difference between these energy estimates and the ones for the fractional Laplacian obtained in \cite{CC1, CC2}. While in our case we can control the energy with a term that does not depend on the exponent $a$ of the weight, for the fractional Laplace problem this is not true when $s$ 
	is small and belongs to the strongly nonlocal range of exponents:
	in particular, the sharp energy estimates proved in \cite{CC2} for $s<\frac{1}{2}$ are not enough to obtain one-dimensional symmetry of special solutions via a Liouville type argument.}
	
	As a consequence of Theorems \ref{DllV}, \ref{globmin}, and \ref{thmono} we deduce the
	following result, which can be seen as the main result of this paper and provides
	the one-dimensional symmetry for minimizers and monotone solutions
	of a three-dimensional water wave problem. 
	
	\begin{theorem}
		\label{monotone}
		Let $f\in C^{1,\gamma}(\R)$, with $\gamma>\max \{0,-a\}$ and let $n=3$. Assume that one of the two following condition is satisfied:
		\begin{itemize}
		\item $v$ is a bounded minimizer for problem \eqref{mainsis};
		\item $v$ is a bounded solution of ~(\ref{mainsis}) such that its trace $u(x)=v(x,0)$ is monotone in some direction.
		\end{itemize}
		Then, there exist $v_0:\R\times(0,1)\to\R$ and $\omega\in S^2$ such that:
		\[
		v(x,y)=v_0(\omega\cdot x,y) \qquad {\mbox{for all }} (x,y)\in\R^3\times(0,1).
		\]
		
		In particular, the trace $u$ of $v$ on $\{y=0\}$ can be written as $u(x)=u_0(\omega\cdot x)$.
	\end{theorem}
	{When~$n=2$ and~$a=0$, the analogue of
	Theorem~\ref{monotone} was established in~\cite{DllV}:
	the improvement
	in our case comes from the enhanced energy estimates
	in Theorem~\ref{thmono}.
	We stress once again that the result in Theorem~\ref{monotone}
	holds true for all $a\in(-1,1)$
	and, as we are going to see, we can perform a unified
proof for all $a\in(-1,1)$, without having to distinguish different regimes.
The fact that the results and the methods are common for all~$a\in(-1,1)$
is indeed a special feature for our problem, and it is related to the fact that
the equation in~\eqref{mainsis} is set in a slab (differently, for instance,
from the cases in~\cite{CC1} and~\cite{CC2}, in which the
energy behavior of minimal solutions is completely different in dependence of~$a$). }

	\subsection{Technical comments and strategy of the proofs}
	It is interesting to point out that the results of this paper
	are new not only in the three-dimensional case, but also in the two-dimensional
	case when~$a\ne0$.
As mentioned above, in the two-dimensional case studied in~\cite{DllV} the energy estimate~\eqref{EB} follows easily by standard elliptic estimates which ensure that the gradient of any bounded solution to~\eqref{mainsis} is bounded. Of course, just using an $L^\infty$ bound on the gradient of the solution, would imply that the energy in cylinders $C_R$ grows like $R^n$, which, for $n=3$ would not be enough to apply Theorem~\ref{DllV}.
	
	Our first energy estimate for minimizers (Theorem~\ref{globmin}) is obtained via a comparison argument, similar to the one used in~\cite{AAC}, based on the construction of a competitor which is constant in the smaller cylinder $C_{R-1}$.
	
	The proof of the energy estimate for monotone solutions is, instead, more involved and it follows the strategy of~\cite{CC1,CC2}, in which a similar estimate is proved for the fractional Laplacian.

	Observe that in Theorem~\ref{thmono} we restrict the statement to the case $n=3$. This is due to the fact that, after taking the limit at $\pm \infty$ in the direction of monotonicity of the solution, we reduce our problem to the classification of stable solutions in one dimension less. Such a classification (more precisely the one-dimensional symmetry and the monotonicity of the limit functions) is known only in dimension  $2$.
	
	We point out that several important differences arise comparing the
	settings in~\cite{DllV} and in~\cite{CC1,CC2}
	with the one considered in this paper. In particular:
	\begin{itemize}
	\item In~\cite{DllV}, only the two-dimensional case is taken into account
	(and only the nonsingular and nondegenerate case~$a:=0$). The lower
	dimensionality assumption is important in~\cite{DllV} since it gives for free
	the appropriated bounds on the energy growth;
	\item In~\cite{CC1,CC2}, the case of purely fractional operators are taken into
	account, while the operators treated here are nonlocal, but nonfractional
	as well, and these special
	features require here, among the other technical bounds,
	new energy estimates and a new set of regularity results, that are tailored
	for the case under consideration. On the other hand, as a byproduct of the sharp
	energy estimates that we find, we are able to obtain symmetry results for all values
	of~$a\in(-1,1)$ (while the energy estimates in~\cite{CC1,CC2}
	cannot be applied beyond the range~$(-1,0]$, thus reflecting the important difference
	between the water wave problem studied here and the fractional Laplace problem
	in~\cite{CC1,CC2}).
	\end{itemize}

	\subsection{Organization of the paper}
	The paper is organized as follows:
	\begin{itemize}
\item In Section 2 we collect some preliminary results on regularity and gradient estimates for solutions to~\eqref{mainsis};
		\item In Section 3 we give the proof of Theorem~\ref{DllV}, which is based on two preliminary results: a characterization of stability (Lemma~\ref{stability}) and a Liouville type theorem (Lemma~\ref{Liouville}). We also deduce directly Corollary~\ref{2D}; 
		\item In Section 4 we prove the energy estimate for minimizers (Theorem~\ref{globmin});
		\item In Section 5 we prove the energy estimate for monotone solutions (Theorem~\ref{thmono}) which needs several ingredients (mainly Lemma~\ref{barv} and Lemma~\ref{minim}).
	\end{itemize}

\section{Regularity results and gradient bounds for solutions to~\eqref{mainsis}}\label{sec.regularity}

In this section we collect some regularity results and gradient estimates for solutions to problem~\eqref{mainsis}.

We start by observing that the weight $y^a$, with $a\in (-1,1)$ belongs to the so-called
Mouckenhoupt class~$A_2$ and hence the theory developed by Fabes, Jerison, Kenig,  and Serapioni~\cite{FKS,FJK} applies to the operator $\mathrm{div}(y^a \nabla)$. 

More precisely in~\cite{FKS,FJK}  a Poincar\'e inequality, a
Harnack inequality, and the 
H\"older regularity for weak solutions of $\mathrm{div}(y^a \nabla)=0$ are established . This theory gives interior regularity for solutions of our problem~\eqref{mainsis}. In the sequel we will need regularity up to the boundary $\{y=0\}\cup \{y=1\}$ and some global $L^\infty$ estimates for the derivatives of solutions to~\eqref{mainsis}. For these results some care is needed, due to the presence of the weight $y^a$.

We define the weighted Sobolev spaces (recall~\eqref{CRCR})
$$L^2(C_R, y^a):=\{v:C_R\rightarrow \R\,|\,y^a v^2 \in L^1(C_R)\}.$$
$$H^1(C_R,y^a):=\{v:C_R\rightarrow \R\,|\,y^a(v^2 + |\nabla v|^2)\in L^1(C_R)\}.$$

In the sequel we will consider the following localized (in the $x$-variable) linear problem
\begin{equation}\label{local}
\begin{cases}
\mathrm{div}(y^a\nabla v)=0 \qquad & \mbox{in}\;\; C_{R}\\
\partial_y v =0 \qquad & \mbox{on}\;\; B_{R}\times \{y=1\}\\
-y^a\partial_y v =g \qquad & \mbox{on}\;\; B_{R}\times \{y=0\}.
\end{cases}
\end{equation}

We start by giving the definition of weak solution for~\eqref{local}.

\begin{definition}
Let $R>0$,  and let $g\in L^1(B_R)$. We say that a function $v\in H^1(C_R,y^a)$ is a \textit{weak solution} of problem~\eqref{local} if
\[\int_{C_R} y^a \nabla v \cdot \nabla \xi\,dx\,dy -\int_{B_R\times \{y=0\}} g\,\xi\,dx =0,\]
for every $\xi \in C_0^\infty (B_R\times [0,1])$.

\end{definition}

Later on, we will need the following duality principle which is the analogue, for our problem, of Proposition 3.6 in~\cite{CY1} (see also~\cite{CS}).

\begin{lemma}\label{duality}
Let $g\in C({\R^n})$, $v\in C^2(\R^{n}\times(0,1))$, and $y^a\partial_y v\in C(\R^{n}\times[0,1])$.  If $v$ is a classical solution of 
\[ \begin{cases}
\mathrm{div}(y^a \nabla v)=0 \qquad & \mbox{in}\,\,\R^{n}\times(0,1)\\
\partial_y v =0 \qquad & \mbox{on}\,\,\R^{n}\times \{y=1\}\\
-y^a\partial_y v = g \qquad & \mbox{on}\,\,\R^{n}\times \{y=0\},
\end{cases}
\]
then the function $w=-y^a\partial_y v$ is a classical solution of the Dirichlet problem
\[
\begin{cases}
\mathrm{div}(y^{-a} \nabla w)=0 \qquad & \mbox{in}\,\,\R^{n}\times(0,1)\\
w =0 \qquad & \mbox{on}\,\,\R^{n}\times \{y=1\}\\
w = g \qquad & \mbox{on}\,\,\R^{n}\times \{y=0\}.
\end{cases}
\]
\end{lemma}

The result
in Lemma~\ref{duality} follows by a simple computation and we refer to~\cite{CS} for its proof.

We can now give a regularity result for the localized linear problem~\eqref{local}.

\begin{proposition}\label{holder}
Let $g\in L^\infty(B_R)$  and let $v$ be a bounded weak solution of ~\eqref{local}.

Then, there exists $\beta\in (0,1)$ (depending only on $n$ and $a$) such that $v\in C^\beta(C_{R/2})$ with the following estimate
$$\|v\|_{C^\beta(\overline{C_{R/2}})}\leq c^1_R,$$

for some $c^1_R$ depending on $n$, $a$, $R$, $\|g\|_{L^\infty(B_R)}$, $\|v\|_{L^\infty(C_R)}$. 

Moreover, we have
\begin{equation}\label{local-v_y}
\|y^a\partial_yv\|_{L^{\infty}(\overline{C_{R/2}})}\leq c^2_R,
\end{equation}
for some $c^2_R$ depending on the same quantities as above. 
\end{proposition} 

\begin{proof}
To prove the $C^{\beta}$ regularity of the solution $v$ in $B_{R/2}\times [0,1)$, we follow the argument used by Cabr\'e and Sire to prove Lemma 4.5 in~\cite{CY1}. 
We need to modify such argument since, in our case, the solution of~\eqref{local}
is not directly related to the fractional Laplacian and a localization
method needs to be exploited.
First, we set $\bar g= g\,\eta$ where $\eta\in C^\infty_0(\R^n)$ is a cut-off function which is identically $1$ in $B_{\frac{3}{4}R}$, so that $\bar g$ is now defined on the whole $\R^n$ and agrees with $ g$ in $B_{\frac{3}{4}R}$.
Let now $\bar v$ be the bounded solution of
\[ \begin{cases}
\mathrm{div}(y^a \nabla \bar v)=0 & \mbox{in }\R^{n+1}_+\\
-y^a \partial_y \bar v=\bar g & \mbox{on }\R^{n}\times \{y=0\},
\end{cases}\]
which is precisely the local problem in the halfspace $\R^{n+1}_+$ associated to the nonlocal equation $(-\Delta)^{\frac{1-a}{2}}\bar u = \bar g$, where $\bar u= \bar v(x,0)$ ($\bar{v}$ is the so-called Caffarelli-Silvestre extension of $\bar{u}$, see~\cite{CS}). By Remark 3.10 in~\cite{CY1}, we have that $\bar v$ is continuous and bounded in $\overline{\R^{n+1}}_+$.  Hence, by Proposition 2.9 in~\cite{Sil}, we have that $\bar u \in C^{\beta} (\R^n)$ for some $\beta \in (0,1)$ depending only on $n$ and $a$.

Let now $\widetilde{v}:=v-\bar v$. Then, in $C_{\frac{3}{4}R}\subset \R^{n+1}_+$, the function $\widetilde v$ solves
\[\begin{cases}
\mathrm{div}(y^a\nabla \widetilde{v})=0 & \mbox{in } C_{\frac{3}{4}R}\\
-y^a \partial_y \widetilde{v}=0 & \mbox{on } B_{\frac{3}{4}R}\times \{y=0\}.
\end{cases}
\]

Since now we have reduced our problem to a problem with zero Neumann condition on $\{y=0\}$ we can do an even reflection of the solution $\widetilde{v}$ with respect to $\{y=0\}$ in order to get a bounded weak solution  of
$$\mathrm{div}(|y|^a \nabla \widetilde{v})=0\quad \mbox{in } B_{\frac{3}{4}R}\times (-1,1).$$

Now, we can apply the regularity theory in~\cite{FKS} (we recall that the weight $|y|^a$
belongs to the Muckenhoupt class~$A_2$) to get that $\widetilde{v}$, and thus $v$, is $C^{\beta}(\overline{B_R}\times [0,1))$ for some $\beta \in (0,1)$ depending only on $n$ and $a$. 

The $C^\beta$ regularity for $v$ up to the top boundary $\{y=1\}$ follows in a standard way, again by even reflection with respect to $\{y=1\}$, observing that the weight $y^a$ is non degenerate for $y=1$ and we have zero Neumann condition on this part of the boundary. This conclude the proof of the first part of the statement.

We now prove~\eqref{local-v_y}. By Lemma~\ref{duality}, the function $w:=-y^a\partial_y v$ solves 
	\[
	\begin{cases}
	\mathrm{div}(y^{-a} \nabla w)=0 \qquad & \mbox{in}\,\,C_R\\
	w =0 \qquad & \mbox{on}\,\,B_R\times \{y=1\}\\
	w = g \qquad & \mbox{on}\,\,B_R\times \{y=0\}.
	\end{cases}
	\]
	
We introduce the function
	\[
	\overline{w}:=P_{\bar{s}}(\cdot,y)\ast\bar{g}
	\]
	where $\bar{g}= g\,\eta$ is defined in the first part of the proof, $\bar{s}$ is such that $1-2\bar{s}=-a$ and $P_{\bar{s}}$ is the Poisson kernel for the fractional Laplacian (see Proposition 3.7 and Remark 3.8 in~\cite{CY1}). We have that $\overline{w}\in L^\infty(\overline{\R^{n+1}_+})$ and satisfies
	
	\[\begin{cases}
		\mathrm{div}(y^{-a} \nabla \overline w)=0 \qquad & \mbox{in}\,\,\R^{n+1}_+\\
	w = \bar g \qquad & \mbox{on}\,\,\R^n\times \{y=0\}.
	\end{cases}
	\]
	
	Now, we can define $\widetilde{w}:=w-\overline{w}$. Arguing as in the first part of the proof, we have that $\widetilde w$ has zero (weighted) Neumann condition on $\{y=0\}$ and hence 
	its odd reflection across~$\{y=0\}$ satisfies
	\[
	\mathrm{div}(|y|^{-a} \nabla \widetilde{w})=0\quad \mbox{in } B_{\frac{3}{4}R}\times (-1,1).
	\]
	Using again the regularity theory in~\cite{FKS} (we recall that the weight $|y|^{-a}$
	belongs to the Muckenhoupt class~$A_2$) we get that $\widetilde{w}$ is $C^{\beta}(\overline{B_{R/2}}\times [0,1))$ with $\beta \in (0,1)$ depending only on $n$ and $a$. Hence the function $w=\widetilde{w}+\overline{w}$ is bounded in $\overline{B_{R/2}}\times [0,1]$ with a bound that only depends on the quantities specified in the statement of the proposition. This concludes the proof.
	
\end{proof}

As a consequence of Proposition \ref{holder}, we get the following estimate for solutions to the semilinaer (localized) problem.
\begin{corollary}\label{semilinear}
Let $f$ be a function in $C^{1,\gamma}(\R)$ , with $\gamma>\max\{0,-a\}$ and let $v$ be a bounded solution of 
\begin{equation}\label{local-f}
\begin{cases}
\mathrm{div}(y^a\nabla v)=0 \qquad & \mbox{in}\;\; C_{R}\\
\partial_y v =0 \qquad & \mbox{on}\;\; B_{R}\times \{y=1\}\\
-y^a\partial_y v =f(v) \qquad & \mbox{on}\;\; B_{R}\times \{y=0\}.
\end{cases}
\end{equation}

Then, there exists $\beta\in (0,1)$ (depending only on $n$ and $a$) such that $v\in C^\beta(C_{R/2})$ with the following estimates
$$\|v\|_{C^\beta(\overline{C_{R/2}})}\leq c^1_R,$$

for some $c$ depending on $n$, $a$, $R$, $\|f\|_{C^{1,\gamma}}$, $\|v\|_{L^\infty(C_R)}$. 

Moreover, we have
\begin{equation}\label{local-v_y-f}
\|y^a\partial_yv\|_{L^{\infty}(\overline{C_{R/2}})}\leq c^2_R,
\end{equation}
for some $c$ depending on $n$, $a$, $R$, $\|f\|_{C^{1,\gamma}}$, $\|v\|_{L^\infty(C_R)}$. 
\end{corollary}

\begin{proof}
It is enough to observe that, since $f\in C^{1,\gamma}$ and $v$ is bounded then $f(v)$ is bounded and hence Proposition \ref{holder} applies to $v$.
\end{proof}

In the following proposition, we establish global gradient estimates for solutions to~\eqref{mainsis} (the semilinear problem in the infinite slab), which will be crucial to establish our main result.
\begin{proposition}\label{grad-estimate}
Let $f$ be a function in $C^{1,\gamma}(\R)$ , with $\gamma>\max\{0,-a\}$ and let $v$ be a bounded solution of~\eqref{mainsis}.

Then, 
\begin{equation}
\label{gradient1}
\|\nabla_x v\|_{L^\infty (\R^n\times [0,1])} + \|y^a\partial_y v\|_{L^\infty(\R^n\times [0,1])}\leq C_1,
\end{equation}
for some $C_1$ depending only on $n$, $a$, $\|f\|_{C^{1,\gamma}}$, $\|v\|_{L^\infty}$.

\end{proposition}

\begin{proof}
	We start with the estimate for $|\nabla_x v|$. Let us define the function
	\[
		v_1(x,y):=\frac{v(x+he,y)-v(x,y)}{\lvert h\rvert^\beta}
	\]
	where $e\in S^{n-1}$,  $h\in\R$ and $\beta$ given by Proposition \ref{holder}
	(and, without loss of generality, possibly reducing~$\beta$, we can assume that~$\beta$
	is of the form~$1/k$ for some integer~$k$). By Corollary~\ref{semilinear}, we have that $v\in C^\beta(\overline{C_{R/2}})$ and hence that $v_1$ is bounded in $C_{R/4}$. In addition, $v_1$ solves
	\begin{equation}\label{local1}
	\begin{cases}
	\mathrm{div}(y^a\nabla v_1)=0 \qquad & \mbox{in}\;\; C_{R/4}\\
	\partial_y v_1 =0 \qquad & \mbox{on}\;\; B_{R/4}\times \{y=1\}\\
	-y^a\partial_y v_1 =\frac{f(v((x+he,0))-f(v(x,0))}{\lvert h\rvert^\beta} \qquad & \mbox{on}\;\; B_{R/4}\times \{y=0\}.
	\end{cases}
	\end{equation}
	Since $f\in C^{1,\gamma}(\R)$ and $v\in C^\beta (\overline{C_{R/2}})$, the right-hand side in the third equation of~(\ref{local1}) is bounded and we can apply Proposition \ref{holder} to $v_1$ in the cylinder $C_{R/4}$. Hence we obtain that $v_1$ is $C^\beta(\overline{C_{R/8}})$ and, using Lemma 5.6 in~\cite{CaffC} and the fact that the direction $e$ is arbitrary, that $v$ is in $C^{2\beta}(\overline{C_{R/8}})$. We have that
	\[
	\|v\|_{C^{2\beta}(\overline{C_{R/8}})}\leq c_R
	\]
	with $c_R$ depending on $n$, $a$, $R$, $\|f\|_{C^{1,\gamma}(\R)}$, $\|v\|_{L^\infty(C_R)}$. 
	
	Now, we can iterate this procedure for a finite number (namely, $k-1$)
	of times such that $k\beta\geq 1$ (this is possible since $\beta$ is a fix strictly positive number depending only on the quantities specified in Proposition \ref{holder}). In this way, we deduce that $\lVert\nabla_x v\rVert_{L^\infty(\overline{C_{R/8^k}})}$ is bounded. Moreover, since problem~\eqref{mainsis} is invariant under translations in the $x$-direction,
	we can obtain uniform estimates for $\lVert\nabla_x v\rVert_{L^\infty}$ in any (closed) cylinder $\overline{C_{R/8^k}(z,0)}=\overline{B_{R/8^k}(z)}\times[0,1]$ with $z\in\R^n$. Observe that the bound $c^1_R$ in Proposition \ref{holder} depends on the radius but not on the center of the balls $B_R$.
	Hence, by a covering argument we obtain the global bound~(\ref{gradient1}).
	
To prove the second part of the statement, we use the bound~\eqref{local-v_y} of Corollary~\ref{semilinear}. Again, after fixing the radius $R=1$  and using a covering argument as before, we deduce that  $\|y^a \partial_y v\|_{L^{\infty}(\R^n\times [0,1])}\leq C_1$, with $C_1$ depending only on $n$, $a$, $\|f\|_{C^{1,\gamma}}$, $\|v\|_{L^\infty}$, which concludes the proof.
\end{proof}

	\section{Proof of Theorem~\ref{DllV}}\label{stable.solutions}
	
	In this section we establish Theorem~\ref{DllV} with a proof based on two main ingredients. The first one is the following characterization of stability, which is the analogue for our problem of Lemma 6.1 in~\cite{CY2}.

%
	
		\begin{lemma}
		\label{stability}
		Let $d$ be a bounded, H\"{o}lder continuous function on $\R^n$. Then
the inequality
		\begin{equation}
		\label{generalstab}
		\int_{\R^n\times (0,1)}y^a\lvert\nabla\eta\rvert^2\,dx\,dy+\int_{\R^n\times \{y=0\}}d(x)\eta^2\,dx\geq0
		\end{equation}
		holds true for any~$\eta\in C^1_0(\R^n\times[0,1])$ if and
only if there exists a H\"{o}lder continuous function $\varphi\in H^1_{\text{loc}}(\R^n\times[0,1],y^a)$, such that
		\begin{align}
		\label{stabsis}
		\begin{cases}
		\mathrm{div}(y^a\nabla\varphi)=0 \qquad &\text{in}\,\,\R^n\times (0,1) \\
-y^a\partial_y\varphi+d(x)\varphi=0 \qquad &\text{on}\,\,\R^n\times \{y=0\} \\
\partial_y\varphi=0 \qquad &\text{on}\,\,\R^n\times\{y=1\}
		\end{cases}
		\end{align}
with
\begin{equation}\label{13bis}
\varphi>0 {\mbox{ in }}\R^n\times [0,1].
\end{equation}
	\end{lemma}
	\begin{proof}
		We first assume the existence of $\varphi$ and we prove~(\ref{generalstab}). Taken a test function
$\eta$ as in the statement of Lemma~\ref{stability}, we can multiply~(\ref{stabsis}) by $\frac{\eta^2}{\varphi}$ and then integrate over $\R^n\times (0,1)$. We obtain:
		\begin{align*}
		0&=\int_{\R^n\times (0,1)}\mathrm{div}(y^a\nabla \varphi)\frac{\eta^2}{\varphi}\\&=-\int_{\R^n\times\{y=0\}}
y^a\partial_y\varphi\frac{\eta^2}{\varphi}-2\int_{\R^n\times (0,1)}y^a\frac{\eta}{\varphi}\nabla\eta\nabla\varphi+\int_{\R^n\times (0,1)}y^a\frac{\lvert\nabla\varphi\rvert^2\eta^2}{\varphi^2}\\
		&\geq-\int_{\R^n\times \{y=0\}}d(x)\eta^2-\int_{\R^n\times (0,1)}y^a\lvert\nabla\eta\rvert^2,
		\end{align*}
where in the last estimate, we have used the boundary data of~\eqref{stabsis}
and Cauchy-Schwarz inequality.
This establishes~(\ref{generalstab}).
		\\ The other implication is more delicate to prove. We first define
		\[
		Q_R(\xi):=\int_{C_R}y^a\lvert\nabla\xi\rvert^2\,dx\,dy+\int_{B_R\times \{y=0\}}d(x)\xi^2\,dx
		\]
		and we take $\lambda_R$ as the infimum of $Q_R(\xi)$ in the set \begin{align*}
		S_R:=\left\{\xi\in H^1(C_R,y^a): \xi\equiv0\,\,\text{on}\,\, \partial B_R\times (0,1), \int_{B_R}\xi^2=1\right\}\\\subset H_0(C_R,y^a)=\{\xi\in H^1(C_R,y^a): \xi\equiv0\,\,\text{on}\,\,  \partial B_R\times (0,1)\}.
		\end{align*}
		{{F}rom} the stability assumption and Definition~\ref{STAB},
we know that~$\lambda_R\geq0$. We want to prove that $\lambda_R$ is
strictly decreasing in $R$, in order to deduce that
		\begin{equation}
		\label{lambdadecr}
		\lambda_R>0.
		\end{equation}
		To show that $\lambda_R$ is decreasing in $R$, we observe that from the hypothesis $\lambda_R$ is nonincreasing and $Q_R$ is bounded below in $S_R$, since $d$ is a bounded function. Now,
		if we take a minimizing sequence $(\xi_k)_k\subset S_R$, we have that $(\nabla\xi_k)$ is uniformly bounded in $L^2(C_R,y^a)$. Using also the compactness of the inclusion $H_0(C_R,y^a)\subset L^2(B_R)$ (see 
		the proof 
		of Lemma~4.1 in~\cite{CY2}), we can state that the infimum of $Q_R$ in $S_R$ is achieved by a function $\varphi_R\in S_R$. We observe also that, up to take $\lvert\varphi_R\rvert$ instead of $\varphi_R$,
		we can choose $\varphi_R\geq0$. We remark that the function $\varphi_R$ solves
		\begin{align}
		\label{4873}
		\begin{cases}
		\mathrm{div}(y^a\nabla\varphi_R)=0 \qquad &\text{in}\,\,C_R \\
-y^a\partial_y\varphi_R+d(x)\varphi_R=\lambda_R\varphi_R \qquad &\text{on}\,\,B_R\times\{y=0\} \\
\partial_y\varphi_R=0 \qquad &\text{on}\,\, B_R\times \{y=1\} \\
\varphi_R=0 \qquad &\text{on}\,\, \partial B_R \times (0,1).
		\end{cases}
		\end{align}
		Hence, from the strong maximum principle, we have that $\varphi_R>0$ in $C_R$.
		\\Now, we take $R_1<R_2$
and our goal is to show that~$\lambda_{R_1}>\lambda_{R_2}$.
Since~$\lambda_{R_1}\ge\lambda_{R_2}$
due to the inclusion of the domains,
we argue by contradiction and suppose that
$\lambda_{R_1}=\lambda_{R_2}$. The strategy is then to
integrate by parts
the quantity
\[\int_{C_{R_1}}\varphi_{R_2}\mathrm{div}(y^a\nabla\varphi_{R_1})\]
in order to obtain a contradiction. Indeed, by using~(\ref{4873}) and the fact
that $\lambda_{R_1}=\lambda_{R_2}$, we find that
		\begin{equation}\label{L10}
		\int_{\partial B_{R_1}\times (0,1)}y^a\varphi_{R_2}
		\frac{\partial\varphi_{R_1}}{\partial\nu}=0.
		\end{equation}
		Since $\varphi_{R_2}>0$ and $\frac{\partial\varphi_{R_1}}{
		\partial\nu}<0$ on $\partial B_R\times (0,1)$, the identity in~\eqref{L10}
		cannot hold true, thus we have reached the desired contradiction.
Hence $\lambda_R$ is strictly
decreasing in $R$ and the proof of~(\ref{lambdadecr}) is complete.
		
		Using the definition of $\lambda_R$ and the fact that~$\lambda_R$
		is strictly positive, we obtain that
		\[
		Q_R(\xi)\geq\lambda_R\int_{B_R}\xi^2\geq-\delta_R\int_{B_R}d(x)\xi^2 \qquad
\quad{\mbox{for all }}\xi\in S_R
		\]
		with $0<\delta_R:=\frac{\lambda_R}{\lVert d\rVert_\infty}$, and therefore
		\begin{equation}
		\label{coercivity}
		Q_R(\xi)\geq\e_R\int_{C_R}y^a\lvert\nabla\xi\rvert^2,
		\end{equation}
		with $\e_R:=1-\frac{1}{1+\delta_R}>0$.
		Now we are able to prove that, fixed~$c_R>0$, there exists a solution $\varphi_R$ to the
		problem
		\begin{align}
		\label{65}
		\begin{cases}
		\mathrm{div}(y^a\nabla \varphi_R)=0 \qquad &\text{in}\,\,C_R \\
-y^a\partial_y \varphi_R+d(x)\varphi_R=0 \qquad &\text{on}\,\,B_R\times \{y=0\} \\
\partial_y \varphi_R=0 \qquad &\text{on}\,\, B_R\times \{y=1\} \\
		\varphi_R=c_R \qquad &\text{on}\,\, \partial B_R\times (0,1).
		\end{cases}
		\end{align}
		Setting~$\varphi_R:=\psi_R+c_R$,
this problem is equivalent to the following one
		\begin{align*}
		\begin{cases}
		\mathrm{div}(y^a\nabla \psi_R)=0 \qquad &\text{in}\,\,C_R \\
-y^a\partial_y \psi_R+d(x)\psi_R+ c_R d(x)=0  \qquad &\text{on}\,\, B_R\times \{y=0\} \\
\partial_y \psi_R=0 \qquad &\text{on}\,\, B_R\times \{y=1\} \\
		\psi_R=0 \qquad &\text{on}\,\, \partial B_R\times (0,1).
		\end{cases}
		\end{align*}
		We notice that
we can solve the latter system
by minimizing in the space $H_0(C_R,y^a)$ the functional
		\begin{eqnarray*}
		D(\xi)&=&\int_{C_R}\frac{1}{2}y^a\lvert\nabla\xi\rvert^2+
\int_{ B_R\times \{y=0\}}\bigg[\frac{1}{2}d(x)\xi^2+c_Rd(x)\xi\bigg]\\&=&
\frac12 Q_R(\xi)+c_R\int_{ B_R\times \{y=0\}} d(x)\xi.
		\end{eqnarray*}
		Since this functional is bounded from below and coercive in~$H_0(C_R,y^a)$,
thanks to~(\ref{coercivity}), and since the inclusion $H_0(C_R,y^a)\subset L^2(B_R)$ is compact, there exists a minimizer of $D$ in $
H_0(C_R,y^a)$.
	
We want now to show that  $\varphi_R$ is strictly positive. To do this, we consider its negative part $\varphi^-_R$. By definition, it vanishes on $ \partial B_R\times(0,1)$, and we can compute that $Q_R(\varphi^-_R)$=0.

Since the first eigenvalue $\lambda_R$ of $Q_R$ is positive, we have that $\varphi^-_R\equiv0$ and so $\varphi_R\geq0$. Hence, using 
the Hopf Lemma (see Lemma~4.11 in~\cite{CY1}), we deduce that $\varphi_R>0$ in $C_R$.

Now that we have found a positive solution of~(\ref{65}), next step is proving that for a fixed  $\delta>0$
		\begin{equation}
		\label{66}
		\sup_{C_R}\varphi_S\leq \widetilde{c}_R \qquad {\mbox{for all }} S>R+\delta,\end{equation}
for some~$\widetilde{c}_R>0$ (we stress that~$\widetilde{c}_R$
depends on~$R$ but not on~$S$).
		To do that, we choose $c_R$ in~(\ref{65}) such that $\varphi_R(0)=1$. \footnote{To see that this is possible, consider $\varphi^1_R$ to be the solution of~\eqref{65} with $c_R=1$. Hence, by the Hopf Lemma, $\varphi^1_R(0)\neq 0$. It is then enough to divide $\varphi^1_R$  by the value $\varphi^1_R(0)$ to get a solution of~\eqref{65} (corresponding to $c_R=(\varphi^1_R(0))^{-1}$) which takes value $1$ at $0$.} Let now  $\varphi_S$ be a solution of~(\ref{65}) in $C_S$, with $S>R+\delta$. We take a family of half balls $\{\mathcal{B}_{r,i}^+\}_i\subset\R^n\times[0,1]$, centered in $(x,0)$ with $x\in \overline{B}_R$, of radius $r\in(0,\frac{\delta}{4})$, in such a way that they cover $\overline{B}_R\times\{y=0\}$ and they have
finite mutual intersection. They are in finite number and we call this number $k$
(such number depends on~$R$, but not on~$S$). Since these balls cover $\overline{B}_R\times\{y=0\}$, there exists $j\in\{1,\dots,k\}$ such that $0\in\mathcal{B}_{r,j}^+$. Since $\mathcal{B}^+_{4r,j}\subset C_S$, we can use the Harnack inequality of Lemma 4.9 in~\cite{CY1} and obtain:
		\[
		\sup_{\mathcal{B}^+_{r,j}}\varphi_S\leq K_R \inf_{\mathcal{B}^+_{r,j}}\varphi_S\leq K_R \qquad \text{for every}\,\,S>R+\delta,
		\]
		where $K_R$ is a constant depending only on $R$. Now, using again the Harnack inequality in every ball $\mathcal{B}^+_{r,i}$ of the covering, and using the fact that the balls intersect two-by-two, we obtain the boundedness of $\varphi_S$ over $B_R\times\{y=0\}$. Thanks to the Neumann condition on
		the top of the slab, we can extend this bound, using the maximum principle, to the whole cylinder of radius $R$, obtaining~(\ref{66}).
		
Using now the regularity result for the linear problem established in Lemma~\ref{holder}, we have a uniform bound on $\|\varphi_S\|_{C^{\beta}(B_{R/2}\times [0,1])}$ for every $S>R+\delta$, therefore we can find a subsequence of $(\varphi_S)$ that converges locally to a function $\varphi\in  C^\beta_{\text{loc}}(\R^n\times[0,1])$ that is positive and solves~(\ref{stabsis}). 

	\end{proof}
	
	\begin{remark}{\rm
		\label{remark.mono}
Let $v$ be a solution of~\eqref{mainsis}	such that $\partial_{x_n}v(x,y)>0$ for any $(x,y)\in \R^n\times [0,1)$. Then, we can apply Lemma~\ref{stability} with the choice $d:=-f'(u)$ and $\varphi:=\partial_{x_n}v$, to deduce that $v$ is stable. The stability of monotone solutions for this kind of problems has already been observed in Lemma 7 in~\cite{DllV} for the case $a=0$. We stress that in this paper we also need the existence of a positive solution to the linearized problem as a necessary (and not only sufficient) condition for stability.
	}\end{remark}
	
%
%

	The second ingredient in the proof of Theorem~\ref{DllV} is the following Liouville-type result, which is the analogue of Theorem 4.10 in~\cite{CY1}.
		For its proof, we refer to Section 4.4 of~\cite{CY1}, where a similar result is proven for some semilinear equations in the half-space. In this case, the adaptation to our framework is straightforward.
	\begin{lemma}
		\label{Liouville}
		Let $\varphi$ be a positive function in $L^\infty_\text{loc}(\R^n\times[0,1])$, $\sigma\in H^1_\text{loc}(\R^n\times[0,1], y^a)$ such that:
		\begin{align*}
		\begin{cases}
		-\sigma \mathrm{div}(y^a\varphi^2\nabla\sigma)\leq0 \qquad &\text{in}\,\,\R^n\times(0,1) \\
y^a\sigma\frac{\partial\sigma}{\partial\nu}\leq0 \qquad &\text{on}\,\,\R^n\times (\{y=0\}\cup\{y=1\})
		\end{cases}
		\end{align*}
		in the weak sense. If in addition:
		\[
		\int_{B_R\times(0,1)}y^a\big(\varphi\sigma\big)^2\leq CR^2
		\]
		holds for every $R>1$, then $\sigma$ is constant.
	\end{lemma}

	We can now give the
	
	\begin{proof}[Proof of Theorem~\ref{DllV}]
	
	 Let $v$ be a stable solution of~\eqref{mainsis}.  Lemma~\ref{stability} implies that there exists a positive function $\varphi$ that solves~(\ref{stabsis}) with $d(x)=-f'(u(x))$. We can define the functions
		\[
		\sigma_i=\frac{\partial_{x_i}{v}}{\varphi} \qquad\text{for}\,\,i=1,\dots, n.
		\]
		Our goal is to prove that they are constant. For every fixed $i$, since $\varphi^2\nabla\sigma_i=\varphi\nabla {v}_{x_i}-v_{x_i}\nabla\varphi$ and using that both $v_{x_i}$ and $\varphi$ satisfy the linearized problem~\eqref{stabsis} with $d(x)=-f'(u(x))$, we deduce
		\[
		\text{div}(y^a\varphi^2\nabla\sigma_i)=0.
		\]

Moreover, using again that $v_{x_i}$ and $\varphi$ satisfy the same linearized problem (in
particular they have the same Neumann condition on $\{y=0\}$), we have
		
		\[
		y^a\sigma_i\partial_y\sigma_i= y^a\frac{v_{x_i}}{\varphi^2}v_{x_iy}-y^a\frac{v_{x_i}^2}{\varphi^2}\frac{\varphi_y}{\varphi}=0 \qquad \text{on}\,\,(\{y=0\}\cup\{y=1\})\times\R^n.
		\]	

Finally, assumption~\eqref{EB} gives		

		\[\int_{C_R}y^a(\varphi \sigma_i)^2 = \int_{C_R}y^a|\partial_{x_i} v|^2\leq CR^2,\]
		and hence we can apply Lemma~\ref{Liouville} to deduce that  $\sigma_i$ is  constant for every $i\in\{1,\dots,n\}$ and we call these constants $c_i$. If $c_i=0$ for every $i=1,\dots n$ then $v$ only depends on $y$ (it is constant in the $x$-variables). Otherwise, the solution $v$ only depends on the  variable $y$ and on the one parallel to the vector $(c_1,\dots,c_n,0)$. We call this horizontal variable $\tilde x$:
		\[
		\widetilde x=\frac{\sum_{i=1}^{n}c_ix_i}{(\sum_{i=1}^{n}c_i^2)^\frac{1}{2}}.
		\]
		
		We have thus proven that the trace $u$ of $v$ on $\{y=0\}$  is a function of only one Euclidean variable and hence can be written in the form
		$$u(x_1,\dots,x_n)=u_0(\widetilde x),$$
		where $u_0$ is a function defined on $\R$.
		We can compute the derivative of $u_0$ to get
		\begin{equation*}
		 u_0'=\bigg(\sum_{i=1}^{n}c_i^2\bigg)^\frac{1}{2}\varphi.
		\end{equation*}
		If $c_i=0$ for every $i\in\{1,\dots,n\}$ then $u_0'\equiv0$, otherwise $u_0'>0$.
	This concludes the proof of Theorem~\ref{DllV}.

	\end{proof}

	From Theorem~\ref{DllV} we can directly deduce the
	
	\begin{proof}[Proof of Corollary~\ref{2D}]
	Since we are considering the case $n=2$, we have from the gradient estimate in Proposition \ref{grad-estimate} that
	\[
	\mathcal{E}_R(v)\leq CR^2.
	\]
	This energy estimate allows us to apply Theorem~\ref{DllV} and to obtain that there exist $v_0:\R\times(0,1)\rightarrow\R$ and $\omega
	\in {\rm S}^1$ such that
	\[
	v(x,y)=v_0 (\omega\cdot x, y)\qquad{\mbox{for any $(x,y)\in\R^2\times(0,1)$.}}
	\qedhere\]
	\end{proof}

	\section{Energy estimate for minimizers}\label{sect.min}
	
	This section is devoted to the proof of
Theorem~\ref{globmin}. Here we
prove the energy estimate~(\ref{estimate.min}) for solutions of~(\ref{eq12})
which minimize the
associated energy, and we argue in an arbitrary dimension $n$
(instead of taking $n=3$ as we are going to do in the next section).
It is worth noting that even if the estimate
has 
no dimensional constraint in the case of minimal solutions, this will not give one-dimensional symmetry of minimizers in further dimensions by applying
our method, unless one is willing to take additional
assumption on the energy growth of the solutions. Indeed, in order to prove Theorem~\ref{monotone} we will use Theorem~\ref{DllV}, which requires the energy in $C_R$ to grow like $R^2$.
	\\We consider $v$ as a bounded minimizer of the functional:
	\[
	\mathcal{E}_R(v)=\frac{1}{2}\int_{C_R}y^a\lvert\nabla v \rvert^2 \,dx\,dy + \int_{B_R\times \{y=0\}} G(u)\,dx,
	\]
	such that $v(x,0)=u(x)$. The function $v$ solves~(\ref{mainsis}), and the potential $G$ is such that $G'(u)=-f(u)$. In particular, the potential~$G$ is naturally
	defined up to an additive constant. To appropriately gauge such constant, we set
	\[
	c_u:=\min \{G(s) \,\,\text{s.t.}\,\, s\in[\text{inf}u,\text{sup}u]\}
	\]
	in order to replace~$G(u)$ with $G(u)-c_u$ and work with a positive potential.
	Moreover, we define $\tau$ as the minimum point of $G$: in this way,
	it holds that~$G(\tau)=c_s$.\\
	\\ As we are going to see, we can directly prove Theorem~\ref{globmin} using
a comparison argument with a suitable choice of the competitor.
	
	\begin{proof}[Proof of Theorem~\ref{globmin}]
		Since $v$ is a minimizer of $\mathcal{E}$, for every admissible competitor $w$ (i.e. $w=v$ on $\partial B_R\times (0,1)$) we have
		\[
		\mathcal{E}_R(v)\leq\mathcal{E}_R(w).
		\]
		We define 
		\begin{equation}
		\label{wdef}
		w(x,y):=\eta_R(x)\tau+(1-\eta_R(x))v(x,y),
		\end{equation}
		where $\eta_R:\R^n\to[0,1]$ is a smooth function that is equal to~$ 1$
		inside $B_{R-1}$ and that vanishes outside~$B_R$. In this way, $w$ is constantly equal to $\tau$ in the cylinder $B_{R-1}\times[0,1]$ and it is equal to $v(x,y)$ on the lateral boundary $\partial B_R\times[0,1]$, so $w$ is an admissible competitor.
		\\ Now we recall the fact that if $u$ is a bounded solution of~(\ref{eq12}), then from Proposition \ref{grad-estimate}
		\begin{equation}
		\label{bdedgr}
		\|\nabla_x v\|_ {L^\infty(\R^n\times[0,1])}\leq C
		\qquad{\mbox{ and }}\qquad
		\|y^a\partial_y v\|_{L^\infty(\R^n\times[0,1])}\leq C.
		\end{equation}

		Using~(\ref{bdedgr}) and the definition~\eqref{wdef} of $w$,  we can control the energy of $w$ as
		\begin{align*}
		\mathcal{E}_R(w)&\leq C\int_{C_R\setminus C_{R-1}}y^a \lvert\nabla v \rvert^2 \,dx\,dy +C \int_{C_R\setminus C_{R-1}}y^a\Big\{|v|^2+\tau^2\Big\}\,dy\\
&\qquad+\int_{(B_R\setminus B_{R-1})\times \{y=0\}}\{ G(w)-c_u\}\,dx\\
&\leq  CR^{n-1}\left(\int_0^1 y^a\,dy + \int_0^1 y^{-a}\,dy\right) + CR^{n-1}\leq C R^{n-1}.
		\end{align*}
This concludes the proof of Theorem~\ref{globmin}.
	\end{proof}

	\section{Energy estimates for monotone solutions}
	The main goal of this section is to prove
the energy estimate of Theorem~\ref{thmono} in dimension three for monotone solutions of~(\ref{eq12}).
We first give the following result, which is the counterpart of Corollary~6
in~\cite{DllV} in presence of a weight.
\begin{lemma}
	\label{mono.extension}
Let $v$ be a solution of~\eqref{mainsis} such that $\partial_{x_n}v(x,0)>0$ for any $x\in \R^n$.

Then, $\partial_{x_n}v(x,y)>0$ for any $(x,y) \in \R^n\times [0,1)$.
\end{lemma}
\begin{proof}
We start by observing that the weak and strong maximum principle hold for weak solutions of problem~\eqref{mainsis}. This follows exactly as in Remark 4.2 in~\cite{CY1}, with the only difference that now we have a Neumann condition on the bottom boundary $\{y=1\}$. In this part of the boundary it is then enough to apply Hopf's Lemma to a possible minimum of the solution $v$ to get the result.
With maximum principles at hand, the proof of the desired result follows exactly the proof of Lemma~5 in~\cite{DllV}.
\end{proof}

Let now $n=3$ and let $v$ be a solution of~\eqref{mainsis} whose trace $u$ on $\{y=0\}$ is monotone
in the last direction $x_3$. By Lemma~\ref{mono.extension},  $v$ is monotone in $x_3$ in the whole slab $\R^3\times [0,1]$, hence
we can define two limit profiles of~$v$
as
	\begin{equation}\label{def.barv}\begin{split}
	\overline{v}(x',y):=\lim_{x_3\to+\infty}v(x,y), \\
	\underline{v}(x',y):=\lim_{x_3\to-\infty}v(x,y),
	\end{split}\end{equation}
	where $x'=(x_1,x_2)$. Notice that $\overline{v}$ and~$\underline{v}$
are defined in~$\R^2\times [0,1]$, namely
we reduced the problem by one dimension by taking the limit in $x_3$. 
This fact allows us to deduce good energy estimates for both $\overline v$ and $\underline v$ which, in turn, implies the one-dimensional symmetry and the monotonicity of $\overline{u}$ and $\underline{u}$ on $\{y=0\}$. 
With these properties for $\underline v$ and $\overline v$, we are able to characterize the potential $G$ associated to equation~(\ref{eq12}) (see Lemma~\ref{aboutG} below).


Then, the proof of the energy estimates for monotone solutions follows by these two steps:
\begin{itemize}
\item if $v$ is a bounded monotone solution to~\eqref{mainsis}, then it is in particular a minimizer for the associated energy in a restricted class of functions $S_R$ (basically functions $\widetilde w$ such that $\underline v \leq \widetilde w\leq \overline v$);
\item the characterization of $G$ implies that the competitor $w$ constructed
in the previous section belongs to the class $S_R$.
\end{itemize}	
	
	Some of these results are well known in the classical case or for the fractional
	Laplacian. Here, we need to prove them for our water waves problem, which
	offers a series of specific complications also due to the fact that the Poisson
	Kernel is not explicit. For the sake of completeness, we are going to explain all the details in this section.

	Using Theorem~\ref{DllV}, we are able to prove some important properties of the two limit profiles. This also gives a characterization of the potential $G$ as a corollary.
	\begin{lemma}
		\label{barv}
		Let $f \in C^{1,\gamma}(\R)$ with $\gamma>\max\{0,-a\}$ and $v$ a bounded solution of~(\ref{mainsis}) whose trace $u$ on $\{y=0\}$ is monotone in $x_3$. 
		
Then $\underline{v}$ and $\overline{v}$ are bounded and stable solutions of~(\ref{mainsis}) with $n=2$, and each of them is either constant or one-dimensional and monotone in the $(x_1,x_2)$-plane.
	\end{lemma}

{{F}rom} Lemma~\ref{barv}, one also obtains:

	\begin{corollary}
		\label{corG}
		Set $m=\inf\underline{u}\leq\widetilde{m}=\sup\underline{u}$ and $\widetilde{M}=\inf\overline{u}\leq M=\sup\overline{u}$ . Then $G>G(m)=G(\widetilde{m})$ in $(m,\widetilde{m})$, $G'(m)=G'(\widetilde{m})=0$ and $G>G(M)=G(\widetilde{M})$ in $(M,\widetilde{M})$, $G'(M)=G'(\widetilde{M})=0$.
	\end{corollary}
	\begin{proof}[Proof of Lemma~\ref{barv}]
		We prove the desired result for $\overline{v}$, clearly the same proof can be replied for $\underline{v}$.
		
		The fact that $\overline{v}$ is a solution follows from seeing it as the limit of a sequence of functions in four variables, that is $\overline{v}(x',y)=\lim_{t\to\infty}v^t(x',x_3,y)$ where $v^t(x',x_3,y)=v(x',x_3 +t,y)$. By Corollary~\ref{semilinear}, we have that $v^t$ uniformly converges up to subsequences to $\overline{v}$ in the $C^\beta$ sense on compact sets.
		\\ Now we want to prove that
		\begin{equation}
		\label{stabbarv}
		\overline{v} \quad \text{is stable}
		\end{equation}
		and then apply Theorem~\ref{DllV}. By Remark~\ref{remark.mono}, we have that if $v$ is a monotone solution of~(\ref{mainsis}) in dimension $n=3$, then $v$ is stable in $\R^3\times (0,1)$, hence
		\begin{equation}
		\label{stabv}
		\int_{\R^3\times (0,1)}y^a\lvert\nabla\xi\rvert^2+\int_{\R^3\times \{y=\}}f'(u)\xi^2\geq0,
		\end{equation}
		for all $\xi\in C_0^\infty(\R^3\times (0,1))$.
		Following an idea in~\cite{CC1}, we define a special test function $\xi$ in order to get the stability inequality for $\overline{v}$. We take $\rho>0$ and a function $\phi_\rho\in C_0^\infty(\R,[0,1])$ such that $\phi_\rho=0$ in $(-\infty,\rho)\cup(2\rho+2,+\infty)$ and $\phi_\rho=1$ in $(\rho+1,2\rho+1)$. For every $\eta\in C_0^\infty(\R^2 \times (0,1))$ we define $\xi(x,y)=\eta(x',y)\phi_\rho(x_3)$. So~(\ref{stabv}) becomes, after dividing it by $\alpha_\rho=\int_{\R}\phi_\rho^2$:
		\begin{align*}
		&\int_{\R^2\times (0,1)}y^a\lvert\nabla\eta\rvert^2\,dx'\,dy+\int_{\R^2\times (0,1)}y^a\eta^2\,dx'\,dy\int_{\R}\frac{(\phi'_\rho)^2}{\alpha_\rho}\,dx_3+\\&-\int_{\R^2\times \{y=0\}}\eta^2\,dx'\,dy\int_{\R}f'(v)\frac{\phi_\rho^2}{\alpha_\rho}\,dx_3\geq0.
		\end{align*}
		When $\rho\to+\infty$ the second term vanishes, because of the definition of $\phi_\rho$. In the third term, thanks to the fact that $f\in C^1(\R)$, we have that $f'(v)\to f'(\overline{v})$, hence
		\[
		\int_{\R^2\times (0,1)}y^a\lvert\nabla\eta\rvert^2-\int_{\R^2\times \{y=0\}} f'(\overline{v})\eta^2\geq0.
		\]
		So we proved~(\ref{stabbarv}).
		In order to conclude that $\overline v$ is one-dimensional and monotone in the $(x_1,x_2)$-plane it is enough to apply Theorem~\ref{DllV}, after observing that, from Proposition \ref{grad-estimate}, $|\nabla_x \overline v|\in L^\infty(\R^2\times [0,1])$ and hence assumption~\eqref{EB} is satisfied.
	\end{proof}

	Before proving Corollary~\ref{corG}, we define the notion of \textit{layer} solution of~(\ref{eq12}) and we give a sufficient condition for the potential $G$ to have a double-well structure.
	\begin{definition}\label{LAYER}
		We say that $v$ is a layer solution for~(\ref{mainsis}) if it satisfies~(\ref{mainsis}),
		\[
		v_{x_n}(x,0)>0 \qquad {\mbox{for all }}x \in \R^n
		\]
		and
		\[
		\lim_{x_n\to\pm\infty}v(x,0)=\pm1 \qquad \text{for every}\,\,x'\in\R^{n-1},
		\]
		where $x=(x',x_n)\in\R^n$.
	\end{definition}
	The following result gives a necessary condition to the existence of layer solutions of~(\ref{mainsis}) with $n=1$.
	\begin{lemma}
		\label{aboutG}
		Let $f\in C^{1,\gamma}(\R)$, with $\gamma>\max\{0,-a\}$ and $G'=f$. Let $v$ be a bounded layer solution of
		\begin{align}
		\label{1dimsis}
		\begin{cases}
		\mathrm{div}(y^a\nabla  v) = 0  &\mbox{ in}\;\;\R^n\times(0,1)\\
		-y^av_y(x,0)=f(v)  &{\mbox{ on}}\;\; \R^n\times \{ y=0\}\\
		v_y(x,1)=0 &{\mbox{ on}}\;\; \R^n\times \{ y=1\}.
		\end{cases}
		\end{align}
		Then
		\begin{equation}\label{DWPX}
		G'(1)=G'(-1)=0. 
		\end{equation}
		
		Moreover, if $n=1$, we also have
		
		\begin{equation}\label{DWPX1}
	G>G(1)=G(-1) \quad \text{in}\,\,(-1,1).
		\end{equation}
		
	\end{lemma}	

Potentials satisfying~\eqref{DWPX},~\eqref{DWPX1} are called
``double-well potentials''.

	\begin{proof}[Proof of Lemma~\ref{aboutG}]
	The proof combines some ideas contained in the proofs of Lemmas 4.8 and 5.3 in~\cite{CY1}.
	
		First we prove that $G'(1)=G'(-1)=0$, which holds in any dimension $n$.  We take $\eta$ smooth and nonnegative with compact support in $B_1\times[0,1)$ and with strictly
positive integral over $B_1$. For $R>0$ we define $\eta_R:=\eta(\frac{x}{R},y)$. We slide the function $v$ in the $x_n$ direction by considering
		\[
		v^t(x,y)=v(x',x_n+t,y),
		\]
		which is also a solution of~(\ref{1dimsis}). So we have:
		\begin{align*}
		0&=\int_{ C_R}\mathrm{div}(y^a\nabla  v^t) \eta_R = \int_{B_R\times\{y=0\}}f(u^t)\eta_R-\int_{C_R}y^a\nabla v^t\cdot\nabla\eta_R\\
		&= \int_{ B_R\times \{y=0\}}f(u^t)\eta_R- \int_{ B_R\times\{y=0\}} y^au^t\partial_y\eta_R+\int_{ C_R}v^t\mathrm{div}(y^a\nabla\eta_R).
		\end{align*}
		We have that the first integral converges as $t\to\infty$ to $f(1)R^n\int_{B_1}\eta$ and the other
two integrals
are bounded by $CR^{n-1}$. Hence $\lvert f(1)\rvert\leq\frac{C}{R}$ for every $R$ and we get $f(1)=0$ by letting $R\to\infty$. In the same way we can prove that $G'(-1)=0$.

	We prove now the second part of the statement.	Let  $n=1$; we  claim that
		\begin{equation}
		\label{1.11}
		\int_{0}^{1}\frac{t^a}{2}\big(v_x^2(x,t)-v_y^2(x,t)\big)\,dt=G(u(x,0))-G(1).
		\end{equation}
		First, we define the function $w$ as:
		\begin{equation*}
		w(x)=\int_{0}^{1}\frac{t^a}{2}\big(v_x^2-v_y^2\big)(x,t)\,dt.
		\end{equation*}
		We remark that $w$ is well defined and bounded thanks to~(\ref{gradient1}). In addition, Proposition \ref{grad-estimate} allows us to derive under the integral sign in the definition of $w(x)$, so we can compute the derivative of $w$ as
		\begin{equation}\label{A}
		\begin{split}
			\partial_x w(x)&=\int_{0}^{1}t^a(v_xv_{xx}-v_yv_{xy})(x,t)\,dt\\&=\displaystyle\lim_{y\rightarrow 0}y^a v_y(x,y)v_x(x,y)=\frac{d}{dx}G(u(x,0)),
		\end{split}
		\end{equation}
		where we have used an integration by parts and the fact that $v$ is a solution of~(\ref{1dimsis}).
		Using~\eqref{A}, we obtain that
		\begin{equation}\label{B}
		w(x)-[G(v(x,0))-G(1)]=C\end{equation}
		for some constant $C$. Our next goal is proving that $C=0$. 
		%

		To this end, first we point out the estimate
		\[\lvert w(x)\rvert \leq C\int_{0}^{1}t^a\lvert\nabla v(x,t)\rvert^2\,dt.
		\]
		We prove now that for every fixed $R>0$
		\begin{equation}
		\label{63}
		\lVert\nabla_x v\rVert_{L^\infty(C_R(x,0))}+\lVert y^a\partial_y v\rVert_{L^\infty(C_R(x,0))}\longrightarrow 0 \qquad\text{as}\,\,x\to+\infty,
		\end{equation}
		where $C_R(x,0)=B_R(x,0)\times(0,1)\subset \R\times(0,1)$.
		
		Suppose by contradiction that there exist $R>0$, $y\in \R$, $\e>0$ and a sequence $t_m\to+\infty$ such that:
		\begin{equation}
		\label{nota}
		\lVert\nabla_x v\rVert_{L^\infty(C_R(x,0))}+\lVert y^a\partial_y v\rVert_{L^\infty(C_R(x,0))}\geq\e \qquad \text{for every}\,\,m.
		\end{equation}
		Notice that $v^{t_m}$ is a solution of~(\ref{1dimsis}) for every $m$.
		Also, the sequence  $v^{t_m}$
		is uniformly bounded. Consequently, we obtain $C^\beta(\overline{C_S})$
		estimates for $v^{t_m}$ from Proposition \ref{holder}, and we stress that these estimates are uniform in $m$ for every $S>0$. {{F}rom} this fact we have that, up to subsequences, $v^{t_m}$ converges to a bounded function $v\in C^\beta_\text{loc}(\R\times[0,1])$ such that
		\[\mathrm{div}(y^a\nabla v)=0.
		\]
		Since $v\equiv1$, we get a contradiction with~(\ref{nota}) and we obtain~(\ref{63}). \\
		Letting now $x \rightarrow +\infty$ and using \eqref{63}, we deduce that $C=0$. Moreover taking the limit
		for $x \rightarrow -\infty$, we also have that
		\begin{equation}
		\label{G1}
		G(1)=G(-1).
		\end{equation}
		Now, we are left with proving that
		\begin{equation*}
		G>G(1) \qquad \text{in}\,\,(-1,1).
		\end{equation*}
		In order to do that, we want to prove that for every $x\in[0,1)$
		\begin{equation}
		\label{1.10}
		\int_{0}^{x}\frac{t^a}{2}\big(v_y^2(t,y)-v_x^2(t,y)\big)\,dt<G(u(0,y))-G(1).
		\end{equation}
		We define for every $y\in\R$ and $x\in[0,1]$
		\[
		\eta(x,y):=\int_{0}^{x}\frac{t^a}{2}\big(v_y^2(t,y)-v_x^2(t,y)\big)\,dt
		\]
		and
		\[\varphi(x,y):=G(u(x,0))-G(1)-\eta(x,y).\]
		First, we observe that $\varphi$ can not be constant. Indeed, since $\eta(x,0)=0$, the fact that $\varphi$ is constant would imply that $G$ is also constant. This would give $f\equiv0$ and so $u$ would be constant, in contradiction with the monotonicity property $u_x>0$.
		\\Using the fact that $\mathrm{div}(y^a\nabla v)=0$, we can compute the derivatives of $\varphi$ as 
		\begin{eqnarray*}
		&&\partial_y \varphi(x,y)=-\frac{y^a}{2}(v_x^2-v_y^2)
		\\{\mbox{and }}&&
		\partial_x\varphi(x,y)=y^av_x(x,y)v_y(x,y).
		\end{eqnarray*}
		Hence, after some computations (see also the proof of Lemma 5.3 in \cite{CY1}) we see that $\varphi(x,y)$ is bounded and satisfies \begin{equation}
		\label{div_a}
		\mathrm{div}(y^a\nabla\varphi(x,y))=-ay^{2a-1}v_x^2(x,y)
		\end{equation} 
		in $\R\times(0,1)$. Our last claim is that
		\begin{equation}\label{C}
		{\mbox{$\varphi$ is strictly positive on $\R\times[0,1)$.}}\end{equation}
		Notice that this claim
		implies~(\ref{1.10}). In order to prove~\eqref{C},
		we assume by contradiction that there exists $(x_0,y_0)$ in $\R\times[0,1)$ such that $\varphi(x_0,y_0)\leq0$. 
		
		Let us divide the proof in two cases, considering at a first attempt $a\geq0$. {{F}rom}~(\ref{div_a}) it follows that
		\begin{equation}
		\label{a.positive}
		\mathrm{div}(y^a\nabla\varphi(x,y))\leq0
		\end{equation} 
		and, using also the Hopf Lemma 4.11 in~\cite{CY1}, we can say that $y_0=0$. Hence there exists $x_0\in\R$ such that
		\[
		G(u(x_0,0))-G(1)\leq0.
		\]
		But $\psi(x)=G(u(x,0))-G(1)$ goes to zero when $x\to\pm\infty$, so we can take $x_0$ as a
		global minimum for $\psi$. It follows that
		\[
		0=\frac{d}{dx}G(u(x_0,0))=\lim_{y\to0}y^av_y(x_0,y)v_x(x_0,y)
		\]
		and from the monotonicity property of $v$, see Remark~\ref{remark.mono}, we have
		\begin{equation*}
		0=-\lim_{y\to0}y^av_y(x_0,y).
		\end{equation*}
		By the maximum principle, the point $(x_0,0)$ is also the
		minimum of $\varphi(x,y)=G(u(x,0))-G(1)-\eta(x,y)$. Since $\varphi$ is the extension of $\psi$ satisfying~(\ref{a.positive}), we have that $(x_0,0)$ is a strict minimum for $\varphi$ and we get a contradiction by considering

		\begin{align*}
		0&>-y^a\partial_y\varphi(x_0,y)_{|_{\{y=0\}}}=y^a\partial_y\eta(x,y)_{|_{\{y=0\}}}\\&
		=\lim_{y\to0}\frac{y^{2a}}{2}\big(v_y^2(x_0,y)-v_x^2(x_0,y)\big)=\lim_{y\to0}\frac{y^{2a}}{2}v_x^2(x_0,y)\geq0.
		\end{align*}
		This contradiction proves~\eqref{C} in the case $a\geq0$. Now we deal with the case $a<0$. First, we compute
		\begin{equation}
		\label{div.minus.a}
		\mathrm{div}(y^{-a}\nabla\varphi(x,y))=-ay^{-1}v_y^2(x,y).
		\end{equation}
		Recalling that we are supposing that a negative minimum of $\varphi$ is achieved at $(x_0,y_0)\in\R^n\times[0,1)$, we want to show that $y_0=0$. Since now $a$ is negative, we have to add an extra term from~(\ref{div.minus.a}) and consider
		\[
		0=\mathrm{div}(y^{-a}\nabla\varphi)+ay^{-1}v_y^2=\mathrm{div}(y^{-a}\nabla\varphi)+\big(ay^{-a-1}\frac{v_y}{v_x}\big)\varphi_x.
		\]
		{{F}rom} the fact that this last operator is uniformly elliptic with continuous coefficients in compact sets of $\R^n\times(0,1)$, it follows that $y_0=0$. Now, we can obtain a contradiction by considering
		\begin{align*}
		0&\geq-\liminf_{y\to0^+}y^{-a}\partial_y\varphi(x_0,y)=y^a\partial_y\eta(x,y)_{|_{\{y=0\}}}\\&
		=\liminf_{y\to0}\frac{1}{2}\big(v_x^2(x_0,y)-v_y^2(x_0,y)\big)=\frac{1}{2}v_x^2(x_0,0)>0.
		\end{align*}
		Notice that we have used also the fact that, from~(\ref{gradient1}), $\lvert v_y(x_0,y)\rvert\leq Cy^{-a}\to0$ as $y\to0^+$. This proves~(\ref{C}) also when $a$ is negative and finishes the proof.
	\end{proof}
	Now that we have characterized the potential $G$ in presence of a layer solution, we are able to deduce Corollary~\ref{corG} from Lemma~\ref{barv}.
	\begin{proof}[Proof of Corollary~\ref{corG}]
		We want to find a layer solution of~(\ref{mainsis}) in order to use the
characterization given by Lemma~\ref{aboutG}. Our candidate
is the function $\overline{w}$ defined as
		\begin{equation}
		\label{9234}
		\overline{w}:=2\bigg(\frac{\overline{v}-\widetilde{M}}{M -\widetilde{M}}\bigg)-1.
		\end{equation}
		We take the function $h(\overline{w}):=\frac{2f(\overline{v})}{M-\widetilde{M}}$, and we call $H$ the potential associated to $h$, so $H'=-h$. Then, $\overline{w}$ is a solution of problem~(\ref{mainsis}) with $n=2$ with $f$ replaced by the new nonlinearity $h(\overline{w})$. By Lemma~\ref{barv} $\overline{w}$ is either constant or
one-dimensional in the $\{x_1,x_2\}$-plane, and it is monotone if it is not constant. According to
Definition~\ref{LAYER}, we have that~$\overline{w}$ is a layer solution of problem~(\ref{mainsis}) (with the new nonlinearity $h$). Now we can apply Lemma~\ref{aboutG}, and obtain that $H$ is a
double-well potential. Restating this result for $G$, we have that $G$ is forced to satisfy $G'(\widetilde{M})=G'(M)=0$ and  $G>G(\widetilde{M})=G(M)$ in  $(\widetilde M, M)$.
		\\ Using $\underline{v}$ instead of $\overline{v}$, we can prove with the same argument that $G'(\widetilde{m})=G'(m)=0$ and $G>G(\widetilde{m})=G(m)$ in  $( m, \widetilde m)$.
	\end{proof}
	
As a final step before giving the proof of Theorem~\ref{thmono}, we need to prove
the following
result, which ensures that if $v$ is a bounded monotone solution 
for problem~\eqref{mainsis}, then it is a minimizer
in a particular class of functions. This result
can be seen as the counterpart of Proposition~6.2 in~\cite{CC2} for the case into consideration here,
in which we have to take into account the singularity and degeneracy of the weights,
the different domain of the equation and the different boundary conditions.

	\begin{lemma}
		\label{minim}
		Let $f\in C^{1,\gamma}(\R)$, with $\gamma>\max \{0,-a\}$, $v$ a bounded solution of~(\ref{mainsis}) with $n=3$, such that its trace $u(x)=v(x,0)$ is monotone in its
third variable. 

Then
		\[
		\mathcal{E}_R(v)\leq\mathcal{E}_R(w)
		\]
		for every $w\in H^1(C_R, y^a)$ such that $w=v$ on $\partial B_R\times(0,1)$ and $\underline{v}\leq w\leq\overline{v}$ in $C_R$.
	\end{lemma}
	\begin{proof}
		The proof of this property follows the proof of Proposition 6.2 in \cite{CC2} and is based on two results:
		\begin{itemize}
			\item[(i)]
			Uniqueness of solutions to the problem
			\begin{align}
			\label{uniqsis}
			\begin{cases}
			\mathrm{div}(y^a\nabla w)=0 \qquad &\text{in}\,\,C_R \\
			w=v \qquad &\text{on}\,\,\partial B_R\times(0,1) \\
			-y^a\partial_y w=f(w) \qquad &\text{on}\,\, B_R\times\{y=0\}\\
			\partial_y w=0 \qquad &\text{on}\,\, B_R\times\{y=1\}\\
			\underline{v}\leq w\leq\overline{v} \qquad &\text{in}\,\, B_R\times(0,1) \\
			\end{cases}
			\end{align}
			
			We give here a proof of this result that uses the idea of sliding the function $v$ in the $x_n$-direction.
			Keeping in mind that $\underline{u}$ and $\overline{u}$ are respectively the trace of $\underline{v}$ and $\overline{v}$ on $\{y=0\}$, let $w$ be a solution of~(\ref{uniqsis}).

By Hopf Lemma (see Lemma~4.11 in~\cite{CY1}) and the maximum principle, we have that
			\begin{equation}
			\label{boundsw}
			\underline{v}< w<\overline{v} \qquad \text{in}\,\, \overline{C_R}.
			\end{equation}
			Now we slide the function $v$ in the direction of monotonicity $x_n$. We take
			\[
			v^t(x,y):=v(x_1,\dots,x_{n-1},x_n+t,y).
			\]
			Since $v_t\rightarrow \overline v$ uniformly  in $\overline C_R$ and by~(\ref{boundsw}), we have that $w<v^t$ in $\overline C_R$  for $t$ large enough. We take $s$ as
			\[
			s:=\inf\big\{t>0\quad\text{s.t.}\quad w<v^t \quad \mbox{in }\overline C_R \big\}.
			\]
			
			We need to prove that
			\begin{equation}
			\label{s0}
			s=0.
			\end{equation}
			Suppose by contradiction that $s>0$. Then we would have $w\leq v^s$ in all $\overline{B_R}\times[0,1] $ and there must be a point~$(\bar{x},\bar{y})$ in which the two functions coincide. But $(\bar{x},\bar{y})\notin \partial B_R\times(0,1)$ because 
			along~$\partial B_R\times(0,1)$ it holds that $w=v$, and we have the monotonicity hypothesis on~$v$. So $(\bar{x},\bar{y})$ must be either in $C_R$ or in $B_R\times\{y=0\}\cup B_R\times\{y=1\}$, but we get a contradiction either with the maximum principle or with the Hopf Lemma
			applied to the positive function $v^s-w$. Hence we proved~(\ref{s0}).

			\item[(ii)] Existence of a minimizer for $\mathcal{E}_R$ in the set:
			\[
			S_R=\{w\in H^1(C_R, y^a) : w\equiv v \,\,\text{on}\,\, \partial B_R\times(0,1),\, \underline{v}\leq w \leq\overline{v}\,\,\text{in}\,\,C_R\}.
			\]
		\end{itemize}
		This result is the analogue of the one obtained in Lemma~4.1 of~\cite{CY2} for layer solutions of the fractional Laplacian. The proof can be adapted by substituting $-1$ and $+1$, which are the limits of the layer solution in~\cite{CY2}, with $\underline{v}$ and $\overline{v}$, which are respectively a subsolution and a supersolution for problem~(\ref{mainsis}).
		
		We already know that $v$ is a solution to problem~(\ref{uniqsis}) and, in view of point (i), we have uniqueness of this solution. So this solution must coincide with the minimizer for $\mathcal{E}_R$ in $S_R$.	
	\end{proof}
		Now we are able to prove the energy estimate of Theorem~\ref{thmono} and to deduce the
		one-dimensional symmetry of monotone solutions from it.
	\begin{proof}[Proof of Theorem~\ref{thmono}]
		We follow the idea in the proof of Theorem 5.2 of~\cite{AAC} and  Theorem 1.3 of~\cite{CC1}, that is we show that the comparison function $\overline{w}$ defined in the previous section satisfies
			\begin{equation}
			\label{winbetween}
			\underline{v}\leq\overline{w}\leq\overline{v}.
			\end{equation}
			In this way, we have that~$\overline{w}$ belongs to~$S_R$, 
			which is the class of functions where $v$ minimizes the energy. We recall that $\overline{w}$ is defined as:
			\begin{equation}
			\overline{w}(x,y)=\tau\eta_R(x)+\big(1-\eta_R(x)\big)v(x,y).
			\end{equation}
			If we prove that $\tau\in[\sup\underline{u},\inf\overline{u}]$, we also have~(\ref{winbetween}) from the maximum principle.

In order to prove this, we use Corollary~\ref{corG}. We set $m=\inf\underline{u}$, $\widetilde{m}=\sup\underline{u}$ and $\widetilde{M}=\inf\overline{u}$, $M=\sup\overline{u}$ .
			We have
			\[
			G>G(m)=G(\widetilde{m}) \quad \text{in}\,\,(m,\widetilde{m}) \quad \text{if $\underline{u}$ is not constant};
			\]
			\[
			G>G(M)=G(\widetilde{M}) \quad \text{in}\,\,(\widetilde{M},M) \quad \text{if $\overline{u}$ is not constant}.
			\]
			Suppose that $m\neq\widetilde{M}$: In all possible cases we have $\widetilde{m}\leq\widetilde{M}$, so there exists $\tau$ in $[\widetilde{m},\widetilde{M}]$ such that $G(\tau)=c_u$, where $c_u$ is the infimum of $G$ in the range of $u$. Hence
			\[
			\sup\underline{v}=\sup\underline{u}=\widetilde{m}\leq \tau\leq\widetilde{M}=\inf\overline{u}=\inf\overline{v}
			\]
			and~(\ref{winbetween}) is proved. Hence, $\overline{w}\in S_R$ and from Lemma~\ref{minim}
			we can conclude that~(\ref{estimate.mono}) holds true.
			
			We have still to consider the special case in which $m=\widetilde{M}$. {{F}rom} Corollary~\ref{corG}, it follows that $G\geq G(m)=G(\widetilde M)$ in $(m,M)$ and all the solutions $v^t$ are obtained by	translation of $v$. We can define a
			new competitor as
			\[
			\widetilde{w}(x,y)=\underline{v}(x,y)\eta_R(x)+\big(1-\eta_R(x)\big)v(x,y),
			\]
			with $\eta_R$ the same cut-off function as before.
			Using again the gradient bounds for $v$ and the fact that $\underline{v}$ depends only on two variables (on $y$ and on one variable in the $(x_1,x_2)$-plane), we deduce that
			\[
			\mathcal{E}_R(\widetilde{w})\leq CR^2.
			\]
			Clearly $\widetilde{w}\in S_R$, so we can use Lemma~\ref{minim} also in this case and obtain the energy estimate for monotone solutions.
	\end{proof}
	{{F}rom} the energy estimates in~(\ref{estimate.min}) and~(\ref{estimate.mono}) we obtain
the one-dimensional symmetry of both minimizers and monotone solutions by a direct application of Theorem~\ref{DllV}.
	\begin{proof}[Proof of Theorem~\ref{monotone}]
		Either if $v$ is a bounded minimizer or is a bounded solution whose trace on $\{y=0\}$ is monotone, we have from Theorem~\ref{thmono} and Theorem~\ref{globmin} that
		\[
		\mathcal{E}_R(v)\leq CR^2.
		\]
		This energy estimate is enough to apply Theorem~\ref{DllV} and to obtain that there exist $v_0:\R\times(0,1)\rightarrow\R$ and $\omega
		\in {\rm S}^2$ such that
		\[
		v(x,y)=v_0 (\omega\cdot x, y)\qquad{\mbox{
				for any $(x,y)\in\R^3\times(0,1)$.}}
		\qedhere\]
	\end{proof}


\section{References}
\begin{biblist}[\normalsize]

	\bib{ABA}{article}{
		author={Abatangelo, N.},
		author={Valdinoci, E.},
		title={Getting acquainted with the fractional Laplacian},
		date={2017},
	}

	\bib{AAC}{article}{
		author={Alberti, G.},
		author={Ambrosio, L.},
		author={Cabr\'{e}, X.},
		title={On a long standing conjecture of E. De Giorgi: symmetry in 3D for general nonlinearities and a local minimality property},
		journal={Acta Appl. Math.},
		volume={65},
		date={2001},
		pages={9-33},
	}
	
	\bib{AC}{article}{
		author={Ambrosio, L.},
		author={Cabr\'{e}, X.},
		title={Entire solutions of semilinear elliptic equations in $\R^3$ and a conjecture of de Giorgi},
		journal	= {J. Amer. Math. Soc.},
		volume={13},
		date={2000},
		pages={725-739},
	}
	
	\bib{BV}{book}{
		author={Bucur, C.},
		author={Valdinoci, E.},
		title={Nonlocal Diffusion and Applications},
		series={Lecture Notes of the Unione Matematica Italiana},
		volume={20},
		publisher={Springer},
		date={2016},
		isbn={978-3-319-28738-6},
	}
	
	\bib{CaffC}{book}{
		author={Cabr\'{e}, X.},
		author ={Caffarelli, L.},
		title ={Fully Nonlinear Elliptic Equations },
		publisher = {Amer. Math. Soc.},
		isbn =      {0821804375,9780821804377},
		year =      {1995},
		series =    {Colloq. Publ. Amer. Math. Soc.},
		edition =   {},
		volume =    {43},
	}
	
	\bib{CC1}{article}{
		author={Cabr\'{e}, X.},
		author={Cinti, E.},
		title={Energy estimates and 1-D symmetry for nonlinear equations involving the half-Laplacian},
		journal={Discrete Contin. Dyn. Syst.},
		volume={28},
		date={2010},
		pages={1179-1206},
	}
	
	\bib{CC2}{article}{
		author={Cabr\'{e}, X.},
		author={Cinti, E.},
		title={Sharp energy estimates for nonlinear fractional diffusion equations},
		journal={Calc. Var. Partial Differential Equations},
		volume={49},
		date={2014},
		pages={233-269},
	}
	
	\bib{CCS}{article}{
		author = {Cabr\'{e}, X.},
		author = {Cinti, E.},
		author = {Serra, J.},
		title = {Stable nonlocal phase transitions},
		journal = {Preprint},
		date = {2017},
	}
	
	\bib{CY1}{article}{
		author={Cabr\'{e}, X.},
		author={Sire, Y.},
		title={Nonlinear equations for fractional Laplacians I: Regularity, maximum principles and Hamiltonian estimates},
		journal={Ann. Inst. H. Poincaré Anal. non Linéaire},
		volume={31},
		date={2014},
		pages={23-53},
	}
	
	\bib{CY2}{article}{
		author={Cabr\'{e}, X.},
		author={Sire, Y.},
		title={Nonlinear equations for fractional Laplacians II: existence, uniqueness, and qualitative properties of solutions},
		journal={Trans. Amer. Math. Soc.},
		volume={367},
		date={2015},
		pages={911-941},
	}
	
	\bib{CSM}{article}{
		author={Cabr\'{e}, X.},
		author={Sol\`{a}-Morales, J.},
		title={Layer solutions in a half-space for boundary reactions},
		journal={Comm. Pure and Appl. Math.},
		volume={58},
		date={2005},
		pages={1678-1732},
	}

	\bib{023495}{article}{
		author = {Cinti, E.},
		author = {Serra, J.},
		author = {Valdinoci, E.}
		title = {Quantitative flatness results and BV-estimates
		for stable nonlocal minimal surfaces},
		journal = {ArXiv e-prints},
		eprint = {1602.00540},
		date = {2016},
		adsurl = {https://arxiv.org/abs/1602.00540},
	}
	
	\bib{CS}{article}{
		author={Caffarelli, L.},
		author={Silvestre, L.},
		title={An extension problem related to the fractional Laplacian},
		journal={Comm. Partial Differential Equations},
		volume={32},
		date={2007},
		pages={1245-1260},
	}
	
	\bib{MR2653742}{article}{
   author={Chermisi, Milena},
   author={Valdinoci, Enrico},
   title={A symmetry result for a general class of divergence form PDEs in
   fibered media},
   journal={Nonlinear Anal.},
   volume={73},
   date={2010},
   pages={695--703},
}

	\bib{MR533166}{article}{
   author={De Giorgi, Ennio},
   title={Convergence problems for functionals and operators},
   conference={
      title={Proceedings of the International Meeting on Recent Methods in
      Nonlinear Analysis},
      address={Rome},
      date={1978},
   },
   book={
      publisher={Pitagora, Bologna},
   },
   date={1979},
   pages={131--188},
   review={\MR{533166}},
}
	
	\bib{DllV}{article}{
		author={de la Llave, R.},
		author={Valdinoci, E.},
		title={Symmetry for a Dirichlet-Neumann problem arising in water waves},
		journal={Math. Res. Lett.},
		volume={16(5)},
		date={2009},
		pages={909-918},
	}
	
	\bib{H}{article}{
		author={Di Nezza, E.},
		author={Palatucci, G.},
		author={Valdinoci, E.},
		title={Hitchhiker's guide to the fractional Sobolev spaces},
		journal={Bull. Sci. Math.},
		volume={136},
		date={2012},
		pages={521-573},
	}
	
	\bib{DFV}{article}{
		author = {Dipierro, S.},
		author = {Farina, A.},
		author = {Valdinoci, E.},
		title = {A three-dimensional symmetry result for a phase transition equation in the genuinely nonlocal regime},
		journal = {ArXiv e-prints},
		eprint = {1705.00320},
		date = {2017},
		adsurl = {},
	}
	

		\bib{XFAH}{article}{
		author = {Dipierro, S.},
		author = {Serra, J.},
		author = {Valdinoci, E.},
		title = {Improvement of flatness for nonlocal phase transitions},
		journal = {ArXiv e-prints},
		eprint = {1611.10105},
		date = {2016},
		adsurl = {},
	}
	
	\bib{DSV}{article}{
		author = {Dipierro, S.},
		author = {Serra, J.},
		author = {Valdinoci, E.},
		title = {Improvement of flatness for nonlocal phase transitions},
		journal = {ArXiv e-prints},
		eprint = {1611.10105},
		date = {2016},
		adsurl = {http://adsabs.harvard.edu/abs/2016arXiv161110105D},
	}

	\bib{Nature07}{article}{
author = {Edwards, A. M.}, 
author = {Phillips, R. A.}, 
author = {Watkins, N. W.}, 
author = {Freeman, M. P.},
author = {Murphy, E. J.},
author = {Afanasyev, V.}, 
author = {Buldyrev, S. V.},
author = {da Luz, M. G.},
author = {Raposo, E. P.},
author = {Stanley, H. E.}, 
author = {Viswanathan, G. M.},
title = {Revisiting L\'evy flight 
search patterns of wandering
albatrosses, bumblebees and deer},
 journal={Nature},
   volume={449},
   date={2007},
   pages={1044--1048},
}
	
	\bib{FS}{article}{
		author = {Figalli, A.},
		author = {Serra, J.},
		title = {On stable solutions for boundary reactions: a {D}e {G}iorgi-type
			result in dimension $4+1$},
		journal = {ArXiv e-prints},
		eprint = {1705.02781},
		date = {2017},
		adsurl = {http://adsabs.harvard.edu/abs/2017arXiv170502781F},
	}

	\bib{FJK}{article}{
		author={Fabes, E. B.},
		author={Jerison, D. S.},
		author={Kenig, C. E.},
		title={The Wiener test for degenerate elliptic equations},
		journal={Ann. Inst. Fourier},
		volume={32},
		date={1982},
		pages={151-182},
	}

	\bib{FKS}{article}{
		author={Fabes, E. B.},
		author={Kenig, C. E.},
		author={Serapioni, R. P.},
		title={The local regularity of solutions of degenerate elliptic equations},
		journal={Comm. Partial Differential Equations},
		volume={7},
		date={1982},
		pages={77-116},
	}
	
	\bib{GT}{book}{
		author={Gilbarg, D.},
		author={Trudinger, N. S.},
		title={Elliptic partial differential equations of second order},
		series={Classics in Mathematics},
		note={Reprint of the 1998 edition},
		publisher={Springer-Verlag},
		date={Berlin, 2001},
		isbn={3.540-41160-7},
	}
	
	\bib{S}{article}{
		author = {Savin, O.},
		title = {Rigidity of minimizers in nonlocal phase transitions},
		journal = {ArXiv e-prints},
		eprint = {1610.09295},
		date={2016},
		adsurl = {http://adsabs.harvard.edu/abs/2016arXiv161009295S},
	}
	
	
	\bib{MR2948285}{article}{
   author={Savin, O.},
   author={Valdinoci, E.},
   title={$\Gamma$-convergence for nonlocal phase transitions},
   journal={Ann. Inst. H. Poincar\'e Anal. Non Lin\'eaire},
   volume={29},
   date={2012},
   pages={479--500},
}
	
	\bib{SV}{article}{
		author={Savin, O.},
		author={Valdinoci, E.},
		title={Some monotonicity results in the calculus of variations},
		journal={J. Funct. Anal.},
		volume={264},
		date={2013},
		pages={2469-2496},
	}

	\bib{Sil}{article}{
		author={Silvestre, L.},
		title={Regularity of the obstacle problem for a fractional power of the Laplace operator},
		journal={Comm. Pure Appl. Math.},
		volume={60},
		date={2007},
		pages={67-112},
		}
	
	\bib{YV}{article}{
		author={Sire, Y.},
		author={Valdinoci, E.},
		title={Fractional Laplacian phase transitions and boundary reactions: a
			geometric inequality and a symmetry result},
		journal={J. Funct. Anal.},
		volume={256},
		date={2009},
		pages={1842--1864},
	}
	
	
\end{biblist}
\end{document}